\documentclass{amsart}
 \usepackage{a4wide}
 \usepackage{amssymb, latexsym}
 \usepackage{amscd}
 \usepackage{bbm}
 \usepackage{bm}
 \usepackage{mathtools}
 \usepackage{color}
 \usepackage{hyperref}
 \usepackage[shortlabels]{enumitem}
\newcommand{\N}{\mathbb N}
\newcommand{\R}{\mathbb R}

\newcommand{\rar}{\rightarrow}

\newtheorem{thm}{Theorem}
\numberwithin{thm}{section}
\newtheorem{lemma}[thm]{Lemma}
\newtheorem{Cor}[thm]{Corollary}

\theoremstyle{definition}
\newtheorem{Rem}[thm]{Remark}
\newtheorem{Def}[thm]{Definition}

\newcommand{\exend}{\hfill $\Diamond$}

\begin{document}
\title{Trimmed sums for observables on the doubling map}
\author[Schindler]{Tanja Schindler}
\address{Australian National University, Research School of Finance, Actuarial Studies and Statistics, 26C Kingsley St,
Acton ACT 2601, Australia}
  \email{\href{mailto:tanja.schindler@anu.edu.au}{tanja.schindler@anu.edu.au}}

\date{\today}

\begin{abstract}
We establish a strong law of large numbers under intermediate trimming for a particular example of 
Birkhoff sums of a non-integrable observable over the doubling map. 
It has been shown in a previous work by Haynes that there is no strong law of large numbers 
for the considered system after removing finitely many summands (light trimming)
even though i.i.d.\ random variables and also some dynamical systems with the same distribution function
obey a strong law of large numbers after removing only the largest summand.
\end{abstract}
\maketitle

\section{Introduction and statement of results}
Considering $T$ an ergodic and measure-preserving  transformation  of  a  probability space
$\left(\Omega,\mathcal{B},\mu\right)$ and an observable $\varphi\colon \Omega\to\mathbb{R}_{\geq 0}$,
there is a crucial difference
in terms of the strong law of large numbers between $\varphi$ being integrable or not.
In the integrable case we obtain by Birkhoff's ergodic theorem that $\mu$-almost surely (a.s.)  
\begin{align*}
 \lim_{n\to\infty}\frac{\sum_{k=1}^n\varphi\circ T^{k-1}}{n}=\int\varphi\mathrm{d}\mu,
\end{align*}
i.e.\ the strong law of large numbers is fulfilled, whereas 
in the case of an observable with infinite expectation, 
Aaronson showed in \cite{aaronson_ergodic_1977} that for
all positive sequences $\left(d_{n}\right)_{n\in\mathbb{N}}$ we have
$\mu$-a.s.
\[
\limsup_{n\rightarrow\infty}\frac{\sum_{k=1}^n\varphi\circ T^{k-1}}{d_{n}}=+\infty\text{ \,\,\,\,\ or \,\,\,\,\,}\liminf_{n\rightarrow\infty}\frac{\sum_{k=1}^n\varphi\circ T^{k-1}}{d_{n}}=0.
\]

However, in certain cases it is possible to obtain a strong law of large numbers after deleting 
a finite number of maximal terms. 
One of the first investigated examples for this situation is the unique continued
fraction expansion of an irrational $x\in\left[0,1\right)$ given by
\[x\coloneqq\frac{1}{c_{1}\left(x\right)+\cfrac{1}{c_{2}\left(x\right)+\ddots}}.
\]
In this case we consider the probability space $([0,1),\mathcal{B},\mu)$
with $\mu$
the Gauss measure given by $d\mu(x)\coloneqq1/\left(\log2\left(1+x\right)\right) d\lambda\left(x\right)$ with  $\lambda$ denoting the Lebesgue measure restricted to $\left[0,1\right)$, 
together with the Gauss map $G\colon [0,1)\to [0,1)$ defined by
\begin{align*}
 G(x)&\coloneqq \begin{cases}0&\text{ if }x=0\\
 \left\{1/x\right\}&\text{ else,}             
              \end{cases}
\end{align*}
where $\{x\}=x-\lfloor x\rfloor$ and $\left\lfloor x\right\rfloor=\max\left\{n\in\mathbb{Z}\colon n\leq x\right\}$.
The observable 
$\chi\colon[0,1)\to\mathbb{N}\cup\{\infty\}$ with 
\begin{align}
 \chi(x)&\coloneqq\left\lfloor 1/x\right\rfloor\label{eq: def chi}
\end{align}
gives then rise to the  stationary (dependent, but $\bm{\psi}$-mixing) process $\chi\circ  G^{n-1}=c_{n}$, $n\in\mathbb{N}$, of the $n$-th continued fraction digit (with this notation including the case of the finite continued fraction expansion of $x\in [0,1)\cap\mathbb{Q}$,
for which we set $1/0\coloneqq\infty$). 
For this example Diamond and Vaaler showed in \cite{diamond_estimates_1986}
that we have $\mu$-a.s.\
\begin{equation}
 \lim_{n\rightarrow\infty}\frac{\sum_{k=1}^n c_k-\max_{1\leq \ell\leq n}c_{\ell}}{n\log n}=\frac{1}{\log2}.\label{eq: DV CF strong law}
\end{equation}

But what happens if we use another transformation $\tau$ instead of the Gauss map $G$?
In this paper we are interested in the doubling map $\tau$ defined as 
$\tau\colon\left[0,1\right) \to\left[0,1\right)$ with 
\begin{align}
 \tau(x)\coloneqq 2x\mod 1.\label{eq: def tau}
\end{align}
It is clear that $\lambda$, the Lebesgue measure restricted to $[0,1)$, is an invariant measure with respect to $\tau$.
Our main interest throughout the paper lies on the dynamical system $\left([0,1), \mathcal{B},\lambda,\tau\right)$
together with the observable $\chi$ given in \eqref{eq: def chi}.
For all $n\in\N$ we set 
\begin{equation}
 a_n\coloneqq\chi\circ\tau^{n-1}\text{\,\,\,\,\, and \,\,\,\,\,}S_n\coloneqq\sum_{k=1}^n a_k.\label{eq: def an Sn}
\end{equation}

Haynes showed in \cite{haynes_quantitative_2014} that the digits $(a_n)$ show a behavior different to the 
continued fraction digits $(c_n)$ in terms of strong laws of large numbers.
To make this more precise we first define our trimmed sums. 
For each $n\in\mathbb{N}$ and $x\in[0,1)$ let $\pi\in\mathcal{S}_{n}$ be a permutation of $\left\{ 1,\ldots,n\right\}$ such that
\[a_{\pi\left(1\right)}(x)\geq a_{\pi\left(2\right)}(x)\geq\ldots\geq a_{\pi\left(n\right)}(x).\]
It is clear that this choice of $\pi$ depends on $n$ and $x$, but for notational convenience in what follows we will suppress this dependence. For any $b\in\N$ we now define
\begin{align}
S_n^{b} & =\sum_{k=b+1}^{n}a_{\pi\left(k\right)}.\label{eq: def Snb}
\end{align}
If $b$ does not depend on $n$, then $S_n^b$ is called a \emph{lightly trimmed sum} and if there exists a sequence of constants $(d_n)$ such that
$\lim_{n\to\infty}S_n^b/d_n=1$ a.s.\ we denote this behavior as a \emph{lightly trimmed strong law}.
From here on we always denote by a.s.\ the almost sure convergence with respect to $\lambda$.

It can be easily concluded from \cite[Theorem 4]{haynes_quantitative_2014} and its proof 
that for any sequence of constants $(d_n)$ and any constant $b\in \mathbb{N}$ we have a.s.\ that
\begin{align*}
\limsup_{n\rightarrow\infty}\frac{S_n^b}{d_{n}}=+\infty\text{ \,\,\,\ or \,\,\,\,}\liminf_{n\rightarrow\infty}\frac{S_n}{d_{n}}=0,
\end{align*}
implying that no lightly trimmed strong law can hold.

The main difference between the continued fraction expansion and the above example is that $\chi$
obeys the structure of the underlying dynamics $G$, but not of $\tau$, 
i.e.\ $\chi$ is constant on each slope of the continued fraction transformation 
while $\tau$ has only one slope on $[0,1/2)$ on which $\chi$ takes different values. 
This results into $(c_n)_{n\in\N}$ having stronger mixing properties, i.e.\ being exponentially $\bm{\psi}$-mixing, 
see \cite{philipp_limit_1988}.
However, the digits $(a_n)_{n\in\N}$ are still $\bm{\alpha}$-mixing, see Section \ref{subsec: mixin prop}. 

We shall note here that this example does not seem exceptional. 
In \cite{aaronson_trimmed_2003} Aaronson gave  general conditions for a lightly trimmed strong law of exponentially $\bm{\psi}$-mixing random variables, 
emended by an example for a mixing dynamical system 
for which a lightly trimmed strong law does not hold 
even though it would hold for i.i.d.\ random variables having the same distribution function.

As there can not be a lightly trimmed strong law
for the dynamical system $\left([0,1), \mathcal{B},\lambda,\tau\right)$
with the observable $\chi$, the next step is to ask if there can be a strong law of large numbers if 
the number of deleted terms depends on $n$.
It can be concluded from \cite[Corollary 1.5]{kessebohmer_strong_2017} that there has to be a sequence 
of natural numbers $(b_n)$ tending to infinity such that $b_n=o(n)$ and a norming sequence $(d_n)$ of positive reals 
such that $\lim_{n\to\infty}S_n^{b_n}/d_n=1$ a.s.
We denote this behavior as an \emph{intermediately trimmed strong law}.
However, this qualitative result does not say anything about a minimal trimming sequence $(b_n)$.

It is the aim of this paper to give precise conditions on the growth of the trimming sequence $(b_n)$ and to give a corresponding norming sequence $(d_n)$ 
such that $S_n^{b_n}/d_n$ fulfills an intermediately trimmed strong law. 

The studied example can be seen as a toy example, 
a very similar example has also been studied in \cite{Gouezel_stable_2008}
proving a stable limit law for the system 
$\left([0,1), \mathcal{B},\lambda,\tau\right)$
with the observable
$\widetilde{\chi}_{\alpha}\colon [0,1)\to \mathbb{R}_{>1}\cup\{\infty\}$, $\alpha\geq 1/2$, with 
$\widetilde{\chi}_{\alpha}(x)=1/x^{\alpha}$ instead of $\chi$. 
As Remark \ref{rem: behav tilde chi} will show the behavior remains the same no matter if we consider $\widetilde{\chi}_1$ or $\chi$.

The results can also be seen as a gap to close in the example of the system $\left([0,1), \mathcal{B},\lambda,\tau\right)$.
If we consider the observable
$\widetilde{\chi}_{\alpha}$ with $\alpha>1$, then $\widetilde{\chi}_{\alpha}$ is integrable 
and we can apply Birkhoff's ergodic theorem. 
If $\alpha<1$, then the optimal trimming sequence $(b_n)$ and the corresponding norming sequence $(d_n)$
for an intermediately trimmed strong law
can be calculated using \cite[Theorem 1.7]{kessebohmer_strong_2017} and coincide with those in the i.i.d.\ case, see \cite[Remark 1.9]{kessebohmer_strong_2017}.
The here considered case closes the gap for $\alpha=1$ and is exceptional as it is the only case which differs significantly from the i.i.d.\ case.

It is also worth mentioning that strong laws of large numbers under trimming are a widely studied topic 
for i.i.d.\ random variables and many limit theorems have already been established in the 80th and 90th. 
Most of the above mentioned limit theorems have predecessors as i.i.d.\ versions, 
for instance 
Mori gave conditions for a lightly trimmed strong law of large numbers in \cite{mori_strong_1976} and \cite{mori_stability_1977}.
Haeusler and Mason and subsequently Haeusler  
developed laws of the iterated logarithm for trimmed sums with regularly varying tail distributions, see \cite{haeusler_laws_1987} and \cite{haeusler_nonstandard_1993}.
From these results it is possible to establish an intermediately trimmed strong law.
An intermediately trimmed strong law for more general distribution functions was also subject in \cite{haeuler_asymptotic_1991} and \cite{kessebohmer_strong_2016}.
However, as can be seen from the above explanation, the behavior in this example differs fundamentally from
the i.i.d.\ case and the methods therefore cannot be transfered immediately.

We will now state our main results and then outline the structure of the paper.

\subsection{Statement of main results}\label{subsec: main results}
%
%
In order to more efficiently state our main theorems, we define two collections of positive real valued functions on the natural numbers,
\begin{align*}
\Psi&\coloneqq\left\{ u:\mathbb{N}\rightarrow\mathbb{R}_{> 0}\colon\sum_{n=1}^{\infty}\frac{1}{u\left(n\right)}<\infty\right\},\qquad\text{and}
\qquad\overline{\Psi}\coloneqq\left\{u:\mathbb{N}\to\mathbb{R}_{> 0}\colon \sum_{n=1}^{\infty}\frac{1}{u\left(n\right)}=\infty\right\}.
\end{align*}
Further, remember our setting of the dynamical system $([0,1),\mathcal{B},\lambda,\tau)$ and the observable $\chi$ 
with the subsequent definitions given in \eqref{eq: def tau}, \eqref{eq: def an Sn}, and \eqref{eq: def Snb}.

Our first result is a positive result which demonstrates that, by only intermediately trimming the sums $S_n$, we can cause the remaining quantities to converge a.s.
\begin{thm}\label{thm: if part}
Suppose that $\psi\in\Psi$ and that, for each $n\in\N$,
\begin{align}
b_n=\left\lfloor\frac{\log\psi\left(\left\lfloor\log n\right\rfloor\right)-\log\log n}{\log 2}\right\rfloor \label{eq: cond bn}.
\end{align}
If
\begin{align}
 \lim_{n\to\infty}\frac{b_n}{\left(\log n\right)^{1/4}}=0,\label{eq: bn o log n}
\end{align}
then we have that
\begin{align}
\lim_{n\to\infty}\frac{S_n^{b_n}}{n\cdot\log n}=1~\text{ a.s.}\label{eq: if lim}
\end{align}
\end{thm}

As an example application of the above theorem, let $\epsilon>0$ and
$\psi\left(n\right)=n\cdot\left(\log n\right)^{1+\epsilon}$.
Then it is not difficult to show that $\psi\in\Psi$ and that the sequence $(b_n)$ defined by \eqref{eq: cond bn} satisfies the estimate
\[b_n=\frac{(1+\epsilon)}{\log 2}\cdot\log\log\log n+o\left(1\right).\]
Since $\epsilon$ is arbitrary, we may conclude using the theorem that, for almost every $x$, if we exclude the largest
\[\left\lfloor\frac{(1+\epsilon)}{\log 2}\cdot\log\log\log n\right\rfloor\]
terms from the sums $S_n$, the remaining quantities will be asymptotic to $n\log n$, as $n$ tends to infinity. We will see from the next result that this is close to best possible.
\begin{thm}\label{thm: only if part}
Suppose that $\psi\in\overline{\Psi}$  and that, for each $n\in\N$,
\begin{align}
b_n=\left\lfloor\frac{\log\psi\left(\left\lfloor\log n\right\rfloor\right)-\log\log n}{\log 2}\right\rfloor.\label{eq: cond bn 1}
\end{align}
Then for almost every $x$ we have that
\begin{align}
\limsup_{n\to\infty}\frac{S_n^{b_n}}{n\cdot\log n}=\infty\label{eq: only if limsup}
\end{align}
and
\begin{align}
\liminf_{n\to\infty}\frac{S_n^{b_n}}{n\cdot\log n}\leq 1.\label{eq: only if liminf}
\end{align}
\end{thm}
For comparison with the previous example, let
$\psi\left(n\right)=n\cdot\log n$.
Then we have that $\psi\in\overline{\Psi}$ and that
\[b_n=\frac{1}{\log 2}\cdot\log\log\log n+o\left(1\right),\]
which is only slightly smaller than the sequences from before. However, for this choice of trimming sequence both \eqref{eq: only if limsup} and \eqref{eq: only if liminf} hold almost surely.

\begin{Rem}\label{rem: behav tilde chi}
The previous statements remain unchanged if we consider $\widetilde{\chi}_1$ with $\widetilde{\chi}_1(x)=1/x$
instead of $\chi$. Let $\widetilde{a}_n=\coloneqq\widetilde{\chi}_1\circ \tau^{n-1}$ and let $\widetilde{S}_n^b$ be defined as $S_n^b$ using $(\widetilde{a}_n)$ instead of $(a_n)$. 
Then we particularly have $0\leq \widetilde{\chi}_1(x)-\chi(x)<1$ for all $x\in[0,1)$
and thus $\big|\widetilde{S}_n^{b_n}-{S}_n^{b_n}\big|\leq n$ and the statements in \eqref{eq: if lim}, \eqref{eq: only if limsup}, and \eqref{eq: only if liminf} do not change if we replace $S_n^{b_n}$ by $\widetilde{S}_n^{b_n}$.
\exend
\end{Rem}

As a companion to above results, we will also consider the distributional properties of the partial sums $S_n$. In this direction we will prove the following theorem.
\begin{thm}\label{thm: weak conv}
We have that
\begin{align}
\lim_{n\to\infty}\frac{S_n}{n\log n}=1\label{eq: Sn log n distr}
\end{align}
in distribution.
\end{thm}
\begin{Rem}
The weak limit theorem is in line with 
the weak limit law for the continued fractions expansion  
$\lim_{n\rightarrow\infty}\sum_{k=1}^nc_k/\left(n\log n\right)=1/\log2$
in probability, see \cite{khintchine_metrische_1935}.
It is likely that the mixing properties of a dynamical system 
have less influence on weak as on strong convergence.
\exend
\end{Rem}

The paper is structured as follows: 
We fist introduce some truncated random variables in Section \ref{sec: truncated rv} which are crucial for the proofs of all three theorems. 
In Section \ref{sec: proof if part} we give the proof of Theorem \ref{thm: if part}
including a skeleton of the proof in Section \ref{subsec: proof main thm} and the details in the subsequent sections.

As we need the statement of Theorem \ref{thm: weak conv} for the proof of Theorem \ref{thm: only if part},
we will first give the proof of Theorem \ref{thm: weak conv} in Section \ref{sec: proof weak conv}
and conclude the paper with a proof of Theorem \ref{thm: only if part} in Section \ref{sec: proof only if}.

\section{Truncated random variables}\label{sec: truncated rv}
For $i,r\ge 1$ define the truncated random variables
\begin{align}
 a_i^{r}\coloneqq a_i\cdot\mathbbm{1}_{\left\{a_i\leq  r\right\}}\,\,\,\text{ and }\,\,\,T_n^{r}\coloneqq \sum_{i=1}^n a_i^{r}.\label{eq: def air Tnr}
\end{align}
Further, denote by $F$ the distribution function of $a_1$, which is given explicitly by
\begin{align}
F\left(x\right)=\begin{cases}1-\frac{1}{1+\lfloor x\rfloor}&\text{if } x\ge 0,\\ 0&\text{if } x<0.\end{cases}\label{eq: F dist func}
\end{align}

With this at hand we are able to compute asymptotic formulas 
for the expectation of the above random variables as follows.
\begin{lemma}
If $\left(f_n\right)$ is a sequence which tends to infinity then we have,
as $n\rar\infty$, that
\begin{align}
\mathbb{E}\left(a_1^{f_n}\right)=\int_0^{f_n}x\mathrm{d}F\left(x\right)&\sim\log f_n,\label{eq: E X tn}
\end{align}
and
\begin{align}
\mathbb{E}\left(T_n^{f_n}\right)=n\cdot\int_0^{f_n}x\mathrm{d}F\left(x\right)&\sim n\cdot\log f_n.\label{eq: E Tn tn}
\end{align}
\end{lemma}
Here and in the following we write $g_n\sim h_n$ for two sequences of reals if $\lim_{n\to\infty}g_n/h_n=1$.
\begin{proof}
It is clear that the distribution function of $a_1^{f_n}$ is given by
\[F_{f_n}(x)=\mathbbm{1}_{[0,f_n]}(x)\cdot (1-F(f_n)+F(x))+\mathbbm{1}_{(f_n,\infty)}(x),\]
therefore we have that
\begin{align*}
\mathbb{E}\left(a_1^{f_n}\right)
&=\int_{\R}x~\mathrm{d}F_{f_n}(x)
=\int_0^{f_n}x~\mathrm{d}F\left(x\right)
=\sum_{k=1}^{f_n}k\cdot\left(\left(1-\frac{1}{k+1}\right)-\left(1-\frac{1}{k}\right)\right)\\
&=\sum_{k=1}^{f_n}\frac{1}{k+1}\sim\log f_n.
\end{align*}
The proof of \eqref{eq: E Tn tn} then follows easily from the fact that, since $\lambda$ is invariant with respect to the map $\tau$, the function $F_{f_n}$ is the distribution function of $a_i^{f_n},$ for any choice of $i,n\ge 1$.
\end{proof}

\section{Proof of Theorem \ref{thm: if part}}\label{sec: proof if part}
\subsection{Proof of main part of Theorem \ref{thm: if part}}\label{subsec: proof main thm}
In this section we will give a skeleton of the proof of Theorem \ref{thm: if part}.
The proof is based on three main lemmas, Lemma \ref{lem: number ak> betak io}, Lemma \ref{lem: sum ak> io}, and Lemma \ref{lem: Tn tn a.s.} which we will state first.
\begin{lemma}\label{lem: number ak> betak io}
For all $\psi\in\Psi$ and all $\epsilon>0$ we have that
\begin{align*}
  \lambda\left(\#\left\{k\leq n\colon a_k>\epsilon\cdot n\cdot\log n\right\}>\left\lfloor \frac{\log\psi\left(\left\lfloor\log n\right\rfloor\right)-\log\log n}{\log 2}\right\rfloor~\text{ i.o.}\right)=0.
\end{align*}
\end{lemma}
Here and in the following we abbreviate "infinitely often" by "i.o."

Next, we give a lemma stating that the large digits do not contribute too much to a truncated sum.
\begin{lemma}\label{lem: sum ak> io}
For $\epsilon>0$ and $t_n=n\cdot\left(\log n\right)^{3/4}$ we have that
\begin{align*}
\lambda\left(\sum_{i=1}^{n}a_i\mathbbm{1}_{\left\{t_n\leq a_i\leq \epsilon\cdot n\cdot\log n\right\}}\geq 3\epsilon\cdot n\cdot\log n~\text{ i.o.}\right)=0.
\end{align*}
\end{lemma}

Finally, the third lemma gives a limiting result about the truncated sum defined in \eqref{eq: def air Tnr}.
\begin{lemma}\label{lem: Tn tn a.s.}
We set $t_n=n\left(\log n\right)^{3/4}$. 
Then
\begin{align*}
 \lim_{n\to\infty}\frac{T_n^{t_n}}{\mathbb{E}\left(T_n^{t_n}\right)}=1~\text{ a.s.}
\end{align*}
\end{lemma}
The proofs of these lemmas are given in Sections \ref{subsec: zero one laws} and \ref{subsec: truncated sum}.

As the last step in this section we give the proof of Theorem \ref{thm: if part}.
\begin{proof}[Proof of Theorem \ref{thm: if part}]
We set again $t_n= n\cdot\left(\log n\right)^{3/4}$ 
and note that $\epsilon\cdot n\cdot \log n> t_n$, for $n$ sufficiently large.
Then we can conclude from Lemma \ref{lem: number ak> betak io} and the definition of $\left(b_n\right)$ that for all $\epsilon>0$
\begin{align}
\lambda\left(S_n^{b_n}\geq T_n^{\epsilon\cdot n\cdot\log n}~\text{ i.o.}\right)=\lambda\left(S_n^{b_n}\geq T_n^{t_n}+\sum_{i=1}^{n}a_i\mathbbm{1}_{\left\{t_n\leq a_i\leq \epsilon\cdot n\cdot\log n\right\}}~\text{ i.o.}\right)=0.\label{eq: Snbn Tn sum io}
\end{align}
Since we have by \eqref{eq: E Tn tn} and Lemma \ref{lem: Tn tn a.s.} that
\begin{align*}
\lambda\left(T_n^{t_n}\geq \left(1+\epsilon\right)n\cdot\log n~\text{ i.o.}\right)=0,
\end{align*}
we can conclude from \eqref{eq: Snbn Tn sum io} and Lemma \ref{lem: sum ak> io} that
\begin{align}
\lambda\left(S_n^{b_n}\geq \left(1+4\epsilon\right) \cdot n\cdot\log n~\text{ i.o.}\right)=0.\label{eq: Snbn>1+3ps nlogn}
\end{align}

On the other hand we have for all $x\in[0,1)$ that
\begin{align*}
 S_n^{b_n}
 &=\sum_{k=1}^na_{\pi(k)}-\sum_{\ell=1}^{b_n}a_{\pi(\ell)}
 \geq\sum_{k=1}^n\left(a_{\pi(k)}\cdot\mathbbm{1}_{\{a_{\pi(k)}\leq t_n\}}\right)-\sum_{\ell=1}^{b_n}\left(a_{\pi(\ell)}\cdot\mathbbm{1}_{\{a_{\pi(\ell)}\leq t_n\}}\right)
 \geq T_n^{t_n}-b_n\cdot t_n
\end{align*}
and
\begin{align*}
\frac{b_n\cdot t_n}{\epsilon\cdot n\cdot \log n}=\frac{b_n}{\epsilon\cdot \left(\log n\right)^{1/4}}
\end{align*}
which tends to zero by \eqref{eq: bn o log n}.
Combining this with the statement of Lemma \ref{lem: Tn tn a.s.} yields for all $\epsilon>0$ that
\begin{align}
 \lambda\left(S_n^{b_n}\leq \left(1+\epsilon\right)\cdot n\log n~\text{ i.o.}\right)=0.\label{eq: Snbn<1+eps nlogn}
\end{align}

Combining \eqref{eq: Snbn>1+3ps nlogn} and \eqref{eq: Snbn<1+eps nlogn} gives the statement of the theorem.
\end{proof}

The rest of Section \ref{sec: proof if part} is structured as follows.
In Section \ref{subsec: dep tau B} we will introduce the induced transformation $\tau_B$ given in \eqref{eq: def tau B}.
Since the random variables $\left(a_n\right)$ highly depend on each other, it is difficult to prove statements directly.
The induced transformation will partly solve this problem as we will see in later sections.
The method to use the induced transformation is classical for piecewise expanding interval maps with 
an indifferent fixed point. 
It goes back to Kakutani and Rokhlin dealing with infinite measure preserving measures, see \cite{kakutani_induced_1943}
and \cite{rohlin_general_1948}.
However, it is also used in the finite measure case to prove limit results on the doubling map taking advantage of the independence structure, see \cite{Gouezel_stable_2008}.

With these techniques at hand we are able to prove Lemma \ref{lem: number ak> betak io} and Lemma \ref{lem: sum ak> io}
in Section \ref{subsec: zero one laws} and Lemma \ref{lem: Tn tn a.s.} in Section \ref{subsec: truncated sum}.

\subsection{Properties of the induced transformation \texorpdfstring{$\tau_B$}{ }}\label{subsec: dep tau B}
We start this section by defining the induced transformation or jump transformation $\tau_B$.
Let $B\coloneqq\left[1/2,1\right)$ and define the first exit time of $B$ by $\phi:\left[0,1\right)\to\mathbb{N}$ as
\begin{align*}
\phi\left(x\right)\coloneqq\inf\left\{ n\in\mathbb{N}_0\colon \tau^nx\in B\right\}+1.
\end{align*}
Furthermore, we define the jump transformation $\tau_B:\left[0,1\right)\to \left[0,1\right)$ by
\begin{align}
\tau_Bx\coloneqq\tau^{\phi\left(x\right)}x.\label{eq: def tau B}
\end{align}
Our strategy is to prove limit results for the sequence of random variables $(\chi\circ\tau_B^{n-1})_n$ instead of $(\chi\circ\tau^{n-1})_n$
and relate the limit results for the first to the latter random variables 
in the end of Sections \ref{subsec: zero one laws} and \ref{subsec: truncated sum}.

The reason for this approach is that the sequence  $\left(\varphi\circ \tau_B^{n-1}\right)_{n\in\N}$ 
is independent for the right choice of $\varphi$
as we will see in Lemma \ref{lem: indep of mth entry} and Corollary \ref{cor: indep of 1st entry}.

Our first lemma reads as follows.

\begin{lemma}\label{lem: lambda tauB invar}
$\tau_B$ is invariant with respect to $\lambda$.
\end{lemma}
It will be proven later in this section.

In order to state our next lemma we define the intervals
\begin{align}\label{eq: def Jji}
 J_{j,i}^m\coloneqq \left[1/2^j-(i+1)/2^{j+m},1/2^j-i/2^{j+m}\right)  
\end{align}
with $j\in\N_0$ and $i\in\left\{0,\ldots, 2^{m-1}-1\right\}$.
For given $m\in\N$
the intervals $(J_{j,i}^m)_{j,i}$ form a partition of $[0,1)$.
Further, denote by $\mathcal{J}^m$ the $\sigma$-algebra generated by $(J_{j,i}^m)_{j,i}$.

For simplicity we also define $\left(J_n\right)_{n\in\mathbb{N}}$ with 
$J_{n}=J_{n,0}^1=\left[1/2^{n+1},1/2^{n}\right)$, for all $n\in\N_0$
and $\mathcal{J}$ the $\sigma$-algebra generated by $\left(J_n\right)$.
Note that $J_0=\left[1/2,1\right)=B$.

Then our next lemma reads as follows.
\begin{lemma}\label{lem: indep of mth entry}
Let, for all $n\in\N$, $\nu_n\colon [0,1)\to\R$ be measurable with respect to $\mathcal{J}^m$. Then, for $u\in\mathbb{N}_0$,
the random variables $(\nu_n\circ\tau_B^{u+(n-1)m})_{n\in\N}$  are mutually independent with respect to $\lambda$. 
\end{lemma}
The next corollary follows immediately from this lemma.
\begin{Cor}\label{cor: indep of 1st entry}
Let, for all $n\in\N$, $\nu_n\colon [0,1)\to\R$ be measurable with respect to $\mathcal{J}$. Then 
the random variables $(\nu_n\circ\tau_B^{(n-1)})_{n\in\N}$  are mutually independent with respect to $\lambda$.
\end{Cor}

\begin{Rem}\label{rem: Omega}
For technical reasons we also introduce $\Omega'\subset [0,1)$ as
\begin{align}
 \Omega'\coloneqq\left\{x\in [0,1)\colon x~\text{ does not have a finite binary expansion}\right\}.\label{eq: def Omega}
\end{align}
The points with a finite binary expansion are exceptional on the one hand as they are finally mapped to zero and $\phi(0)=\infty$.
On the other hand, we have for $x\in J_n\cap \Omega'$ that $\chi(x)\in [2^n,2^{n+1}-1)$ 
but $2^{-n-1}\in J_n$ and $\chi\left(2^{-n-1}\right)=2^{n+1}\notin[2^n,2^{n+1}-1)$.

Still, $\Omega'$ is of full measure and thus the above lemmas, Lemma \ref{lem: lambda tauB invar},
Lemma \ref{lem: indep of mth entry}, and Corollary \ref{cor: indep of 1st entry},
are still valid if we restrict ourselves to $\Omega'$ and the respective $\sigma$-algebras $\mathcal{J}\cap \Omega'$ and
$\mathcal{J}^m\cap \Omega'$. 

To clarify our calculations we will sometimes write $\lambda\lvert_{\Omega'}\left(A\right)=\lambda\left(\Omega'\cap A\right)$
even though $\lambda\lvert_{\Omega'}\left(A\right)=\lambda\left(A\right)$ holds for all $\lambda$-measurable sets $A$.
\exend
\end{Rem}

We will prove the previous lemmas by a general approach considering interval maps as in the following lemma.
\begin{lemma}\label{lem: xi invariance}
  Let $\left[0,1\right)$ be partitioned into $(W_i)_{i\in I}$ with $W_i=\left[c_i,d_i\right)$
  and $I$ a finite or countable index set 
  and let $\mathcal{W}$ be the $\sigma$-algebra generated by those intervals.
 Further, for all $i\in I$, let $\xi\colon \left[0,1\right)\to \left[0,1\right)$ be defined on $W_i$ by
 \begin{align}
  \xi\vert_{W_i}x\coloneqq-\frac{c_i}{d_i-c_i} +\frac{1}{d_i-c_i}\cdot x.\label{eq: rep indep of beta}
 \end{align}
 Then $\lambda$ is $\xi$-invariant.
\end{lemma}
In other words the map $\xi$ maps each of the intervals $W_i$ to the full interval
and on each interval the function $\xi$ has a constant positive gradient. 
One example is the doubling map itself with the partition $[0,1/2)$ and $[1/2,1)$. 
We note here that these maps are generalised L\"uroth maps, 
in this generalised form first studied in \cite{barrionuevo_ergodic_1996}, 
but see also
\cite[Chapter 1.4.1]{kessebohmer_infinite_2016}, \cite[Chapter 2.2]{dajani_ergodic_2002}.
A proof of \ref{lem: xi invariance} is given in \cite[Theorem 1]{barrionuevo_ergodic_1996}.
(In some literature it is assumed that the partitioning intervals are ordered by size, but this assumption does not change the proof.)

Also note that the above definition implies only that $\xi$ is a.s.\ defined on $\left[0,1\right)$.
Having for example the partition into the intervals $\left[1/2^k,1/2^{k-1}\right)$ with $k\in\mathbb{N}$ gives 
$\bigcup_{k=1}^{\infty}\left[1/2^k,1/2^{k-1}\right)=\left(0,1\right)$.
For the following we will ignore the nullset of points which might not been defined.

%
Furthermore, the above defined maps have the following handy property:
\begin{lemma}\label{lem: indep of beta}
Let $\xi$ be given as in \eqref{eq: rep indep of beta} with the corresponding partition $(W_i)_{i\in I}$.
If, for all $n\in\N$, the map $\varphi_n\colon\left[0,1\right)\to \mathbb{R}$ is measurable with respect to $\mathcal{W}$,
then the random variables $(\varphi_n\circ \xi^{n-1})_{n\in\N}$ are mutually independent with respect to $\lambda$.
\end{lemma}
\begin{proof}
\cite[Lemma 1]{barrionuevo_ergodic_1996} states that the random variables 
$\big(\widetilde{\varphi}_n\circ \xi^{n-1}\big)_{n\in\N}$
are mutually independent with respect to $\lambda$, where $\widetilde{\varphi}_n(x)=i$ if $i\in W_i$.

The proof remains the same if we replace $\widetilde{\varphi}_n$ by a sequence of more general mappings $(\varphi_n)$
which are measurable with respect to $\mathcal{W}$.
\end{proof}

With the above two lemmas at hand we are able to prove Lemma \ref{lem: lambda tauB invar} and Lemma \ref{lem: indep of mth entry}.
\begin{proof}[Proof of Lemma \ref{lem: lambda tauB invar}]
For $x\in J_n$, we only have to show the representation  
\begin{align}
 \tau_Bx
 =-\frac{2^{-n-1}}{2^{-n}-2^{-n-1}}+\frac{1}{2^{-n}-2^{-n-1}}\cdot x
 = -1+2^{n+1} x.\label{eq: rep tauB}
\end{align}
Applying Lemma \ref{lem: xi invariance} immediately gives the statement of Lemma \ref{lem: lambda tauB invar}.
We note here that $\phi\left(0\right)=\infty$ and $\tau_B0$ is not defined, i.e.\ $\tau_B$ is only almost surely defined on $\left[0,1\right)$.
On $B$ we have that $\tau x=2x\mod 1=-1+2x$.
Obviously, for $x\in B$ we have that $\phi\left(x\right)=1$ and thus 
$\tau_Bx=\tau x=-1+2x$.
In general, if $x\in J_n$, then $\phi\left(x\right)=n+1$, i.e.\ we have that $2^{n}\cdot x\in B$ and thus 
$\tau_Bx=\tau^{\phi\left(x\right)}x=\tau^{n+1}\left(x\right)=\tau \left(2^{n}x\right)=-1+2^{n+1}x$.
\end{proof}

\begin{proof}[Proof of Lemma \ref{lem: indep of mth entry}]
It is enough to show independence of $(\nu_n\circ\tau_B^{u+\left(n-1\right)m})^{-1}(\mathcal{A})$, $n\in\N$, 
where $\mathcal{A}$ is an $\cap$-stable set which generates $\mathcal{J}^m$, see for example 
\cite[Theorem 2.16]{klenke_probability_2007}, i.e.\ we might use  $\mathcal{A}=\left(J_{j,i}^m\right)_{j,i}$. 
If we define $A_{u,m,j,i}^n\coloneqq\{x\colon \tau_B^{u+\left(n-1\right)m} x\in J_{j,i}^m\}$, 
it is enough to prove independence for sequences of sets $(A_{u,m,v(n)}^n)_{n\in\N}$ for all possible functions $v\colon \mathbb{N}\to\{(j,i)\colon j\in\N_0, i\in\{0,\ldots, 2^{m-1}-1\}\}$.

Furthermore, we have by the $\tau_B$-invariance of $\lambda$, see Lemma \ref{lem: lambda tauB invar},
that 
\begin{align*}
 \lambda\left(A_{u,m,v(1)}^1\cap \ldots \cap A_{u,m,v(n)}^n\right)
 &= \lambda\left(\tau_B^{-u}\left(A_{0,m,v(1)}^1\cap \ldots \cap A_{0,m,v(n)}^n\right)\right)\\
 &=\lambda\left(A_{0,m,v(1)}^1\cap \ldots \cap A_{0,m,v(n)}^n\right).
\end{align*}
The last equation implies that we only have to prove independence of $(A_{0,m,v(n)}^n)_{n\in\N}$
or independence of $(\nu_n\circ\tau_B^{(n-1)m})_{n\in\N}$.

Our strategy is to apply Lemma \ref{lem: indep of beta} to the transformation $\tau_B^{m}$. 
For doing so we define for given $m\in\mathbb{N}$ an auxiliary partition of $\left[0,1\right)$ by the intervals
\begin{align*}
 L_{i_1,\ldots, i_m}\coloneqq\left[a_{i_1,\ldots,i_m}, b_{i_1,\ldots,i_m}\right)
\end{align*}
with
\begin{align*}
a_{i_1,\ldots,i_m}&\coloneqq\frac{1}{2^{i_1}}+\ldots+\frac{1}{2^{i_1+\ldots+i_{m-1}}}+\frac{1}{2^{i_1+\ldots+i_m}},\\
b_{i_1,\ldots,i_m}&\coloneqq\frac{1}{2^{i_1}}+\ldots+\frac{1}{2^{i_1+\ldots+i_{m-1}}}+\frac{1}{2^{i_1+\ldots+i_m-1}},
\end{align*}
and $i_1,\ldots, i_m\in\N$. 
Obviously, $L_{i_1}=J_{i_1-1}$
and $L_{i_1,\ldots, i_m}\subset L_{i_1,\ldots, i_{m-1}}$.
We aim to show that the partition $(L_{i_1,\ldots,i_m})$ is such that 
$\tau_B^{m}$ is a map of the form given in \eqref{eq: rep indep of beta}
with respect to this partition.
To show this, let $x\in L_{i_1,\ldots,i_m}$. Then we have that $x\in J_{i_1-1}$. 
Applying the representation in \eqref{eq: rep tauB} gives 
that $\tau_Bx\in L_{i_2,\ldots,i_m}$ and thus $\tau_Bx\in J_{i_2-1}$.
Hence, applying the representation in \eqref{eq: rep tauB} repeatedly and using
an induction argument shows, for $x\in L_{i_1,\ldots,i_m}$,
\begin{align}
 \tau_B^{m}x= -1-2^{i_m}-2^{i_{m}+i_{m-1}}-\ldots-2^{i_{m}+\ldots+i_2}+2^{i_1+\ldots+i_m}x.\label{eq: tauBm+1}
\end{align}

On the other hand we have that 
\begin{align*}
 \frac{1}{b_{i_1,\ldots,i_m}-a_{i_1,\ldots,i_m}}=2^{i_1+\ldots+i_m}
\end{align*}
and 
\begin{align*}
 -\frac{a_{i_1,\ldots,i_m}}{b_{i_1,\ldots,i_m}-a_{i_1,\ldots,i_m}}
 =-1-2^{i_m}-2^{i_{m}+i_{m-1}}-\ldots-2^{i_{m}+\ldots+i_1}.
\end{align*}
Hence, the representation in \eqref{eq: tauBm+1} coincides with the representation in \eqref{eq: rep indep of beta},
allowing us to apply Lemma \ref{lem: indep of beta}. This yields that the random variables $(\nu_n\circ \tau_B^{u+(n-1)m})_{n\in\N}$ are independent
if, for all $n\in\N$, $\nu_n$ is measurable with respect to the $\sigma$-algebra $\mathcal{L}$ generated by $(L_{i_1,\ldots,i_m})$.
Noting that $\mathcal{J}^m$ is a sub-$\sigma$-algebra of $\mathcal{L}$ gives the statement of the lemma.
\end{proof}

\subsection{Zero-one laws concerning the number of large entries \texorpdfstring{$a_n$}{an}}\label{subsec: zero one laws}
In this section we will prove Lemma \ref{lem: number ak> betak io} and Lemma \ref{lem: sum ak> io}.
We will start with a set of definitions and lemmas relevant for the proof of these lemmas.
For the following we set
\begin{align*}
\beta_n\left(x\right)&\coloneqq \chi\circ \tau_B^{n-1}\left(x\right)
\end{align*}
and
\begin{equation}
 \phi_n\left(x\right)\coloneqq\sum_{k=0}^{n-1}\phi\circ \tau_B^k.\label{eq: def phi n}
\end{equation}
Then
the following lemma gives the relation between $\left(a_n\right)$ and $\left(\beta_n\right)$.
\begin{lemma}\label{lem: rel beta a}
 We have that $a_1=\beta_1$ and,
 for $k\in\mathbb{N}_{\geq 2}$, that 
 \begin{align*}
  \beta_k=a_{\phi_{k-1}\left(x\right)+1}.
 \end{align*}
\end{lemma}
\begin{proof}
The first statement follows immediately from the definition of $(a_n)$ and $(\beta_n)$.

By definition $\tau_B=\tau^{\phi(x)}$.
Obviously, $\beta_2=\chi\circ \tau_B=\chi\circ\tau^{\phi(x)}=a_{\phi(x)+1}=a_{\phi_1(x)+1}$.
Using an induction argument
assuming $\tau_B^{k-2}=\tau^{\phi_{k-2}}$
gives then
\begin{align*}
 \tau_B^{k-1}(x)
 =\tau^{\phi\left(\tau_B^{k-2}\left(x\right)\right)}\circ\tau_B^{k-2}\left(x\right)
 =\tau^{\phi\left(\tau_B^{k-2}\left(x\right)\right)}\circ\tau^{\phi_{k-2}}\left(x\right)
 =\tau^{\phi_{k-1}\left(x\right)}\left(x\right)
\end{align*}
and thus
\begin{align*}
 \beta_k\left(x\right)
 &=\chi\circ \tau_B^{k-1}(x)
 =\chi\circ \tau^{\phi_{k-1}\left(x\right)}\left(x\right)
 =a_{\phi_{k-1}\left(x\right)+1}.
\end{align*}
\end{proof}

The following two lemmas, Lemma \ref{lem: Xi> io 0} and Lemma \ref{lem: Xi> 2x io 0} give zero-one laws for large entries $\beta_i$.
\begin{lemma}\label{lem: Xi> io 0}
We have for all $\psi\in\Psi$ that
\begin{align*}
\lambda\left(\#\left\{i\leq n\colon \beta_i\geq n\cdot \psi\left(\left\lfloor \log n\right\rfloor\right)\right\}\geq 1~\text{ i.o.}\right)=0.
\end{align*}
\end{lemma}
In order to prove this lemma we will start with a technical lemma.

\begin{lemma}\label{lem: log gamma log tilde gamma}
Let $\psi\in\Psi$. Then there exists $\omega\in\Psi$ such that
\begin{align}
\omega\left(\left\lfloor \log_2 n\right\rfloor \right)\leq \psi\left(\left\lfloor \log n\right\rfloor\right).\label{phi psi a b}
\end{align}
\end{lemma}

\begin{proof}
We define $\omega:\mathbb{N}\rightarrow\mathbb{R}_{>0}$ as
\begin{align}
\omega\left(n\right)=\min\left\{ \psi\left(\left\lfloor n\cdot \log 2\right\rfloor +j\right)\colon j\in\left\{ 0,1\right\} \right\}.\label{omega n}
\end{align}
Recall that $\psi\in\Psi$. Then for the functions
$\overline{\psi}:\mathbb{N}\to\mathbb{R}_{>0}$ and $\widetilde{\psi}:\mathbb{N}\to\mathbb{R}_{>0}$ given by $\overline{\psi}\left(n\right)=\psi\left(\left\lfloor \kappa\cdot n\right\rfloor \right)$
with $\kappa>0$ and $\widetilde{\psi}\left(n\right)=\min\left\{ \psi\left(n\right),\psi\left(n+1\right)\right\} $
it holds that $\widetilde{\psi},\overline{\psi}\in\Psi$.
Hence, $\omega\in\Psi$.
Applying $\left\lfloor \log_2 n\right\rfloor $ on $\omega$ in \eqref{omega n} yields
\begin{align*}
\omega\left(\left\lfloor \log_2 n\right\rfloor \right)=\min\left\{ \psi\left(\left\lfloor \left\lfloor \frac{\log n}{\log 2}\right\rfloor\cdot\log 2 \right\rfloor +j\right)\colon j\in\left\{ 0, 1 \right\} \right\} .
\end{align*}
Since on the one hand we have 
\begin{align}
\left\lfloor \left\lfloor \frac{\log n}{\log 2}\right\rfloor\cdot\log 2 \right\rfloor
\geq\left\lfloor\left(\frac{\log n}{\log 2}-1\right)\cdot\log 2 \right\rfloor
\geq \left\lfloor \log n\right\rfloor -1\label{eq: omega estim 1}
\end{align}
and on the other hand
\begin{align}
\left\lfloor \left\lfloor \frac{\log n}{\log 2}\right\rfloor\cdot\log 2 \right\rfloor
\leq \left\lfloor  \frac{\log n}{\log 2} \cdot\log 2 \right\rfloor
=\left\lfloor \log n \right\rfloor,\label{eq: omega estim 2}
\end{align}
we have that 
\begin{align*}
 \min\left\{ \psi\left(\left\lfloor \left\lfloor \frac{\log n}{\log 2}\right\rfloor\cdot\log 2 \right\rfloor +j\right)\colon j\in\left\{ 0,1 \right\} \right\}
 \leq \psi\left(\left\lfloor\log n\right\rfloor\right)
\end{align*}
and \eqref{phi psi a b} follows.
\end{proof}

\begin{proof}[Proof of Lemma \ref{lem: Xi> io 0}]
Let $\psi\in\Psi$ be given.
By Lemma \ref{lem: log gamma log tilde gamma} there exists $\widetilde{\psi}\in\Psi$ such that 
$\widetilde{\psi}\left(\left\lfloor \log_2 m\right\rfloor\right)\leq \psi\left(\left\lfloor\log m\right\rfloor\right)$, for all $m\in\mathbb{N}$. 
Let for the following 
$\widetilde{\psi}$ fulfill this inequality.
Since $\lambda$ is $\tau_B$-invariant, see Lemma \ref{lem: lambda tauB invar}, we have, using the distribution function in \eqref{eq: F dist func}, 
for all $i,k\in\N$,
\begin{align*}
\lambda\left(\beta_i\geq 2^{\left\lfloor k+\log_2\widetilde{\psi}\left(k\right)\right\rfloor}\right)
=\lambda\left(\chi\geq 2^{\left\lfloor k+\log_2\widetilde{\psi}\left(k\right)\right\rfloor}\right)
=\frac{1}{2^{\left\lfloor k+\log_2\widetilde{\psi}\left(k\right)\right\rfloor}}
<2^{-k+1}\cdot\widetilde{\psi}\left(k\right).
\end{align*}
Next we notice that 
\begin{align*}
 \lambda\left(\#\left\{i\leq 2^{k+1}\colon\beta_i\geq 2^{\left\lfloor k+\log_2\widetilde{\psi}\left(k\right)\right\rfloor}\right\}\geq 1\right)
 &\leq \sum_{i=1}^{2^{k+1}}\lambda\left(\beta_i\geq 2^{\left\lfloor k+\log_2\widetilde{\psi}\left(k\right)\right\rfloor}\right)
 <\frac{4}{\widetilde{\psi}\left(k\right)}.
\end{align*}
Since $\widetilde{\psi}\in\Psi$ this implies
\begin{align*}
\sum_{k=1}^{\infty}\lambda\left(\#\left\{i\leq 2^{k+1}\colon\beta_i\geq 2^{\left\lfloor k+\log_2\widetilde{\psi}\left(k\right)\right\rfloor}\right\}\geq 1\right)<\infty
\end{align*}
and applying the first Borel-Cantelli lemma yields
\begin{align*}
\lambda\left(\#\left\{i\leq 2^{k+1}\colon\beta_i\geq 2^{\left\lfloor k+\log_2\widetilde{\psi}\left(k\right)\right\rfloor}\right\}\geq 1~\text{ i.o.}\right)=0.
\end{align*}
If we define the sequence of sets $\left(I_k\right)_{k\in\N}$ as 
\begin{align}
 I_k\coloneqq\left[2^{k},2^{k+1}-1\right]\cap \mathbb{N}, \label{eq: Kk}
\end{align}
then we have for every $n\in I_{k}$ that
\begin{align}
\lambda\left(\#\left\{i\leq n\colon\beta_i\geq 2^{\left\lfloor k+\log_2\widetilde{\psi}\left(k\right)\right\rfloor}\right\}\geq 1~\text{ i.o.}\right)=0.\label{eq: estim betai 2k}
\end{align}
On the other hand, if $n\in I_{k}$, we obtain by our choice of $\widetilde{\psi}$ that
\begin{align*}
2^{\left\lfloor k+\log_2\widetilde{\psi}\left(k\right)\right\rfloor}
&\leq 2^{ k+\log_2\widetilde{\psi}\left(k\right)}
\leq n\cdot\widetilde{\psi}\left(\left\lfloor \log_2 n\right\rfloor\right)
\leq n\cdot\psi\left(\left\lfloor\log n\right\rfloor\right).
\end{align*}
Applying this estimation on \eqref{eq: estim betai 2k} yields the statement of the lemma.
\end{proof}

\begin{lemma}\label{lem: Xi> 2x io 0}
For $t_n=n\cdot \left(\log n\right)^{3/4}$ we have that
\begin{align*}
\lambda\left(\#\left\{i\leq n\colon \beta_i\geq t_n\right\}\geq 2~\text{ i.o.}\right)=0.
\end{align*}
\end{lemma}

\begin{proof}[Proof of Lemma \ref{lem: Xi> 2x io 0}]
For $n\in I_{k}$ with $I_k$ as in \eqref{eq: Kk} we have 
\begin{align*}
t_n
=n\cdot \left(\log n\right)^{3/4}
\geq 2^k\cdot\left(\log 2^k\right)^{3/4}
=2^{k+3/4\cdot \log_2\left(\log 2\cdot k\right)}
\geq 2^{k+\left\lfloor3/4\cdot \log_2\left(\log 2\cdot k\right)\right\rfloor}.
\end{align*}
This implies 
\begin{align*}
 \bigcup_{n\in I_k}\#\left\{i\leq n\colon \beta_i\geq t_n\right\}
 &\subset \#\left\{i\leq 2^{k+1}\colon \beta_i\geq 2^{k+\left\lfloor3/4\cdot \log_2\left(\log 2\cdot k\right)\right\rfloor}\right\}.
\end{align*}

For the following we will restrict our space to $\Omega'$ defined in \eqref{eq: def Omega}.
Our strategy is to consider, for $k\in\N$,
\begin{align}
 \lambda\left(\bigcup_{n\in I_k}\#\left\{i\leq n\colon \beta_i\geq t_n\right\}\geq 2\right)
 &\leq \lambda\lvert_{\Omega'}\left(\#\left\{i\leq 2^{k+1}\colon \beta_i\geq 2^{k+\left\lfloor3/4\cdot \log_2\left(\log 2\cdot k\right)\right\rfloor}\right\}\geq 2\right)\notag\\
 &=\sum_{\ell=2}^{2^{k+1}}\lambda\lvert_{\Omega'}\left(\#\left\{i\leq 2^{k+1}\colon \beta_i\geq 2^{k+\left\lfloor3/4\cdot \log_2\left(\log 2\cdot k\right)\right\rfloor}\right\}=\ell\right)\label{eq: estim beta geq 2}
\end{align}
and to calculate the summands independently. 
We have, for each $i\in\N$, $q\in \N_0$,  $\{\beta_i\geq 2^{q}\}=\{\mathbbm{1}_{[0,2^{-q}]}\circ \tau_B^{i-1}=1\}$.
Clearly, $\mathbbm{1}_{[0,2^{-q}]}$ is $\mathcal{J}\cap\Omega'$-measurable and thus, 
by Corollary \ref{cor: indep of 1st entry}
and Remark \ref{rem: Omega}
the events $\{\beta_i\geq 2^{q}\}_{i\in\N}$ are independent. 
Since, by Lemma \ref{lem: lambda tauB invar}, $\lambda$ is additionally $\tau_B$-invariant, we obtain for these summands
\begin{align*}
\lambda\lvert_{\Omega'}\left(\#\left\{i\leq 2^{k+1}\colon \beta_i\geq 2^{k+\left\lfloor3/4\cdot \log_2\left(\log 2\cdot k\right)\right\rfloor}\right\}= \ell\right)
&\leq \binom{2^{k+1}}{\ell}\cdot \lambda\left(\chi\geq 2^{k+\left\lfloor3/4\cdot \log_2\left(\log 2\cdot k\right)\right\rfloor}\right)^{\ell}.
\end{align*}
Using the distribution function of $\chi$ given in \eqref{eq: F dist func} gives
\begin{align*}
\lambda\lvert_{\Omega'}\left(\#\left\{i\leq 2^{k+1}\colon \beta_i\geq 2^{k+\left\lfloor3/4\cdot \log_2\left(\log 2\cdot k\right)\right\rfloor}\right\}= \ell\right)
&< 2^{(k+1)\ell}\cdot \left(\frac{1}{2^{k+\left\lfloor3/4\cdot \log_2\left(\log 2\cdot k\right)\right\rfloor}}\right)^{\ell}\\
&< 2^{(k+1)\ell}\cdot \left(\frac{2}{2^{k+3/4\cdot \log_2\left(\log 2\cdot k\right)-1}}\right)^{\ell}\\
&=\left(\frac{8}{\left(\log 2\cdot k\right)^{3/4}}\right)^{\ell}.
\end{align*}
Hence, applying this on \eqref{eq: estim beta geq 2} and using the geometric series formula implies
\begin{align*}
 \lambda\left(\bigcup_{n\in I_k}\#\left\{i\leq n\colon \beta_i\geq t_n\right\}\geq 2\right)
 &\leq \sum_{\ell=2}^{2^{k+1}}\left(\frac{8}{\left(\log 2\cdot k\right)^{3/4}}\right)^{\ell}
 < \sum_{\ell=2}^{\infty}\left(\frac{8}{\left(\log 2\cdot k\right)^{3/4}}\right)^{\ell}\\
 &=\left(\frac{8}{\left(\log 2\cdot k\right)^{3/4}}\right)^{2}\cdot\left(1-\frac{8}{\left(\log 2\cdot k\right)^{3/4}}\right)^{-1}.
\end{align*}
We have that 
$1-8/\left(\log 2\cdot k\right)^{3/4}\geq 1/2$ if $k$ is sufficiently large. 
This implies 
\begin{align*}
 \lambda\left(\bigcup_{n\in I_k}\#\left\{i\leq n\colon \beta_i\geq t_n\right\}\geq 2\right)
 < \frac{64}{\left(\log 2\cdot k\right)^{3/2}},
\end{align*}
for $k\in\N$ sufficiently large,
which implies 
\begin{align*}
\sum_{k=1}^{\infty}\lambda\left(\bigcup_{n\in I_k}\#\left\{i\leq n\colon \beta_i\geq t_n\right\}\geq 2\right)<\infty.
\end{align*}
Applying the first Borel-Cantelli lemma yields
\begin{align*}
\lambda\left(\bigcup_{n\in I_k}\#\left\{i\leq n\colon \beta_i\geq t_n\right\}\geq 2~\text{ i.o.}\right)=0.
\end{align*}
Noting that $(I_k)$ is a partition of the natural numbers gives the statement of the lemma.
\end{proof}

As a last step before we can start with the proof of the main lemmas we need the following technical lemma.
\begin{lemma}\label{lem: ai+k}
 Assume, for some $i\in\N$, $x\in\Omega'$, that $a_i(x)=r\geq 2$. Then we have for all $j\in \N_{\leq \lfloor \log_2 r\rfloor}$ that
 $a_{i+j}(x)=\left\lfloor r/2^j \right\rfloor$.
\end{lemma}
\begin{proof}
 The statement $a_i(x)=r$ is equivalent to $\chi\circ \tau^{i-1}(x)=r$. 
 The definition of $\chi$ implies  $\tau^{i-1}(x)\in (1/(r+1), 1/r]$ 
 and $\tau^{i-1}(x)\in J_{\lfloor \log_2 r\rfloor-1}$ (taking into account that $x$ we restrict the space to $\Omega'$).
 If $j\leq \lfloor \log_2 r\rfloor$,
 then $\tau^j\circ\tau^{i-1}(x)= 2^j\cdot \tau^{i-1}(x)$, see the proof of Lemma \ref{lem: lambda tauB invar}. 
 Hence,
 \begin{align*}
  a_{i+j}=\chi\left(2^j\cdot \tau^{i-1}(x)\right)
  \in\left[\left\lfloor\frac{r}{2^j}\right\rfloor,\left\lfloor\frac{r+1}{2^j}\right\rfloor\right).
 \end{align*}
 Since $\left\lfloor(r+1)/2^j\right\rfloor-\left\lfloor r/2^j \right\rfloor\leq 1$ and $a_{i+j}$ can only attain 
 natural values we have $a_{i+j}=\left\lfloor r/2^j \right\rfloor$.
\end{proof}

Finally, we prove the two main lemmas of this section. 
\begin{proof}[Proof of Lemma \ref{lem: number ak> betak io}]
For ease of notation set $i(1)\coloneqq 1$ and $i(k)\coloneqq i(k,x)\coloneqq \phi_{k-1}(x)+1$, for all $k\in\N_{\geq 2}$.
Lemma \ref{lem: rel beta a} implies $\beta_{k}(x) = a_{i(k)}(x)$.
Since $\phi(x)\geq 1$, we have that $i(k)\geq k$ implying 
\begin{align}
 \#\left\{k\leq n\colon a_k>\epsilon\cdot n\cdot\log n\right\}
&\leq \#\left\{k\leq i(n)\colon a_k>\epsilon\cdot n\cdot\log n\right\}.\label{eq: kleq n kleq phin}
\end{align}
If we set 
$Y_{k,n} \coloneqq \sum_{j = 0}^{i(k+1) - i(k) - 1} \mathbbm{1}_{\{a_{i(k) + j}> \epsilon \cdot n \cdot \log n\}}$,
then \eqref{eq: kleq n kleq phin} implies 
\begin{align}
 \#\left\{k\leq n\colon a_k>\epsilon\cdot n\cdot\log n\right\}
 &\leq \sum_{\ell=1}^n Y_{k,n}\label{eq: number k = sum Ykn}
\end{align}
and by Lemma \ref{lem: Xi> io 0} we have eventually almost surely, for each $k\leq n$, $\overline{\psi}\in\Psi$,
\begin{align}
 Y_{k,n} &=\sum_{j = 0}^{i(k+1) - i(k) - 1} \mathbbm{1}_{\{\epsilon \cdot n \cdot \log n< a_{i(k) + j}\leq n\cdot\overline{\psi}\left(\left\lfloor\log n\right\rfloor\right)\}}.\label{eq: Ykn= sum}
\end{align}

Let us restrict ourselves again to $\Omega'$ given in \eqref{eq: def Omega}.
If, on $\Omega'$, $a_n=r$ with $r\geq 2$, then by Lemma \ref{lem: ai+k}
\begin{align*}
a_{i+\left\lfloor\log_2 r/\ell\right\rfloor+1}
\leq \frac{r}{2^{\left\lfloor\log_2(r/\ell)\right\rfloor+1}}
 < \frac{r}{2^{\log_2(r/\ell)}}
 \leq \ell.
\end{align*}
Setting $r=n\cdot\overline{\psi}\left(\left\lfloor \log n\right\rfloor\right)$ and $\ell=\epsilon\cdot n\cdot\log n$
and applying this on \eqref{eq: Ykn= sum} yields, for all $k\leq n$,
\begin{align}
 Y_{k,n}\leq \left\lfloor\log_2\frac{n\cdot\overline{\psi}\left(\left\lfloor \log n\right\rfloor\right)}{\epsilon\cdot n\cdot\log n}\right\rfloor\label{eq: Ykn leq}
\end{align}
eventually almost surely.

Furthermore, applying Lemma \ref{lem: Xi> 2x io 0} and noting that
$t_n=n\cdot \left(\log n\right)^{3/4}<\epsilon n \log n$, for $n$ sufficiently large, yields that 
eventually almost surely at most one summand on the righthand side of \eqref{eq: number k = sum Ykn} can be non-zero. 
Combining this with \eqref{eq: Ykn leq} 
yields for all $\overline{\psi}\in\Psi$
\begin{align*}
  \lambda\left(\#\left\{k\leq n\colon a_k>\epsilon\cdot n\cdot\log n\right\}>\left\lfloor\log_2\frac{n\cdot\overline{\psi}\left(\left\lfloor \log n\right\rfloor\right)}{\epsilon\cdot n\cdot\log n}\right\rfloor~\text{ i.o.}\right)=0.
\end{align*}
If we set $\overline{\psi}\left(n\right)\coloneqq\psi(n)/\epsilon$ for given $\psi\in\Psi$,
we obtain $\overline{\psi}\in\Psi$ and
\begin{align}
 \left\lfloor\log_2\frac{n\cdot\overline{\psi}\left(\left\lfloor \log n\right\rfloor\right)}{\epsilon\cdot n\cdot\log n}\right\rfloor
 &=\left\lfloor \frac{\log\psi\left(\left\lfloor\log n\right\rfloor\right)-\log\log n}{\log 2}\right\rfloor.\label{eq: log 2 phi}
\end{align}
Combining this consideration with \eqref{eq: log 2 phi} gives the statement of the lemma.
\end{proof}

\begin{proof}[Proof of Lemma \ref{lem: sum ak> io}]
We use the same notation of $i$ introduced at the beginning of the proof of Lemma \ref{lem: number ak> betak io}
giving $\beta_k(x) = a_{i(k)}(x)$.

Since $i$ is strictly increasing we have in particular $i(k) \geq k$ and thus we have on $\Omega'$ that
\begin{align}
\sum_{k=1}^n a_k \cdot\mathbbm{1}_{\{t_n\leq a_k\leq \epsilon \cdot n \cdot \log n\}}
  &\leq \sum_{k=1}^{i(n)} a_k \cdot\mathbbm{1}_{\{t_n\leq a_k\leq \epsilon \cdot n \cdot \log n\}}\notag\\
  &= \sum_{k=1}^n \sum_{j = 0}^{i(k+1) - i(k) - 1} a_{i(k) + j}\cdot\mathbbm{1}_{\{t_n\leq a_{i(k)+j}\leq \epsilon \cdot n \cdot \log n\}}
  = \sum_{k=1}^n Z_{k,n}.\label{eq: sum ai < sum Z}
\end{align}

Furthermore, for $j\in[0, i(k+1) - i(k))\cap\N$ and $x\in\Omega'$ we obtain from Lemma \ref{lem: ai+k}
\[
a_{i(k) + j}(x) = \lfloor a_{i(k)}(x)/2^j \rfloor = \lfloor \beta_k(x)/2^j \rfloor.
\]

Let $1 \leq k \leq n$ and let $\ell \in \mathbb{N}$ be minimally chosen such that $\lfloor \beta_k/2^\ell \rfloor \leq \epsilon \cdot n \cdot \log n$.
Then we have on $\Omega'$
\begin{align*}
Z_{k,n} 
 &= \sum_{j = 0}^{i(k+1) - i(k) - 1} a_{i(k) + j}\cdot\mathbbm{1}_{\{t_n\leq a_{i(k) + j}\leq \epsilon \cdot n \cdot \log n\}} 
 = \sum_{j = 0}^{i(k+1) - i(k) - 1} \left\lfloor \frac{\beta_k}{2^j} \right\rfloor
    \cdot\mathbbm{1}_{\{t_n\leq \lfloor\beta_k/2^j\rfloor\leq \epsilon \cdot n \cdot \log n\}} \\
 &= \sum_{j = \ell}^{i(k+1) - i(k) - 1} \left\lfloor \frac{\beta_k}{2^j}\right\rfloor
    \cdot\mathbbm{1}_{\{t_n\leq \lfloor\beta_k/2^j\rfloor\leq \epsilon \cdot n \cdot \log n\}} 
 \leq \sum_{j = \ell}^{i(k+1) - i(k) - 1} \left\lfloor \frac{\beta_k}{2^j} \right\rfloor\\
 &< \sum_{j = \ell}^{\infty} \frac{\beta_k}{2^j}
 = \beta_k\cdot 2^{-\ell+1}.
\end{align*}
By the choice of $\ell$ we have that 
\begin{align}
 Z_{k,n}\leq 2\left(\epsilon \cdot n \cdot \log n+1\right)\leq 3\epsilon \cdot n \cdot \log n,\label{eq: Zin leq}
\end{align}
for $n$ sufficiently large.

Furthermore, Lemma \ref{lem: Xi> 2x io 0} implies that eventually almost surely $\#\{k\leq n\colon Z_{k,n}>0\}\leq 1$.
Combining this with \eqref{eq: sum ai < sum Z} and \eqref{eq: Zin leq} gives the statement of the lemma.
\end{proof}

\subsection{Limit results for the truncated sum \texorpdfstring{$T_n^{r}$}{ }}\label{subsec: truncated sum}
This section is devoted to the proof of Lemma \ref{lem: Tn tn a.s.}.
For the following we define $\eta\colon\left[0,1\right)\to \mathbb{R}_{\geq 0}$ and its truncated version $\eta^{r}$, for $r\geq 1$, by
\begin{align*}
 \eta\left(x\right)&\coloneqq\sum_{k=0}^{\phi\left(x\right)-1}\chi\circ \tau^k\left(x\right)\,\,\,\text{ and }\,\,\,
 \eta^{r}\coloneqq\sum_{k=0}^{\phi\left(x\right)-1}\chi^{r}\circ \tau^k\left(x\right).
\end{align*} 
Furthermore, we define for $m,j\in\N$, and $i\in\{0,\ldots,2^{j-1}-1\}$
\begin{align}
 y_{j,i}&\coloneqq y_{j,i}(m)\coloneqq\frac{2^{j+m+1}}{2^m-i-1}\,\,\,\text{ and }
 \,\,\,
 z_{j,i}\coloneqq z_{j,i}(m)\coloneqq\frac{2^{j+m+1}}{2^m-i}-j-1.\label{eq: def y z}
\end{align}
Further, we define the observables $v_m,w_m\colon [0,1)\to\mathbb{R}_{>0}$ 
as well as their truncated versions $v_{m}^{r},w_m^{r}\colon [0,1)\to\mathbb{R}_{>0}$ for $r\geq 1$ 
as functions piecewise constant on $J_{j,i}^m$ such that for $x\in J_{j,i}$
\begin{align}
 v_m\left(x\right)&=y_{j,i},&
 w_m\left(x\right)&=z_{j,i},\notag\\
 v_m^{r}\left(x\right)&=y_{\min\{j,\left\lfloor\log_2r\right\rfloor\},i},&
 w_m^{r}\left(x\right)&=z_{\min\{j,\left\lfloor\log_2r\right\rfloor-1\},i}.\label{eq: def wmr vmr}
 \end{align}
Those observables obey the $\sigma$-algebra $\mathcal{J}^m$ defined in \eqref{eq: def Jji} and thereafter.
Since, for given $u\in\mathbb{N}_0$, each of the sequences of random variables $(\nu\circ \tau_B^{m\cdot (n-1)+u})_{n}$
with $\nu\in\{v_m,w_m,v_m^r,w_m^r\}$
are independent, see Lemma \ref{lem: indep of mth entry}, it is easier to prove statements for those sequences of random variables
instead of $\left(\eta\circ \tau_B^{n-1}\right)_{n}$, see the proof of Lemma \ref{lem: Sn* allg}.

In the next lemmas, Lemma \ref{lem: rel w eta ell}, Lemma \ref{lem: EW vm}, and Lemma \ref{lem: EW wm},
we will work out the relation between the observables $v_m^{r}$, $w_m^{r}$, and $\eta^r$,
showing that the first two give an approximation of the latter one.

With these results at hand we are able to prove the subsequent Lemma \ref{lem: Sn* allg}, 
an analogous statement to Lemma \ref{lem: Tn tn a.s.} for the Birkhoff sum $\sum_{k=0}^{n-1}\eta^{2 t_n}\circ \tau^{k}$
instead of $\sum_{k=0}^{n-1}\chi^{t_n}\circ \tau^{k}$.

In the last part of this section we will then relate the limiting results for the Birkhoff sums $\sum_{k=0}^{n-1}\eta^{2 t_n}\circ \tau^{k}$
and $\sum_{k=0}^{n-1}\chi^{2 t_n}\circ \tau^{k}$ proving Lemma \ref{lem: Tn tn a.s.}.
\begin{lemma}\label{lem: rel w eta ell}
We have for all $x\in\Omega'$, all $m\in\mathbb{N}$, and all $r\geq 1$ that
\begin{align}
 w_m^{r}\left(x\right)\leq \eta^{r}\left(x\right)\leq v_m^{r}\left(x\right).\label{eq: rel w eta}
\end{align}
\end{lemma}
\begin{proof}
We start by showing
\begin{align}
 w_m\left(x\right)\leq \eta\left(x\right)\leq v_m\left(x\right),\label{eq: rel w eta wo ell}
\end{align}
for all $x\in\Omega'$.
We will first give a connection between $\chi\left(x\right)$
and $\eta\left(x\right)$.
Let $x\in J_{n}\cap\Omega'$, then we have on the one hand $\phi\left(x\right)=n+1$, see the proof of Lemma \ref{lem: lambda tauB invar}.
On the other hand, $\left\lfloor\log_2\chi\left(x\right)\right\rfloor=\left\lfloor\log_2 \left\lfloor 1/x\right\rfloor\right\rfloor=n$
giving $\phi\left(x\right)=\left\lfloor\log_2\chi\left(x\right)\right\rfloor+1$.

Using Lemma \ref{lem: ai+k} gives
\begin{align}
\chi\left(x\right)+\sum_{k=1}^{\phi\left(x\right)-1}\left(\chi\left(x\right)\cdot 2^{-k}-1\right) 
\leq \eta\left(x\right)
\leq \sum_{k=0}^{\phi\left(x\right)-1}\chi\left(x\right)\cdot 2^{-k}.\label{eq: eta beta estim}
\end{align}
Since
$\phi\left(x\right)-1=\left\lfloor\log_2 \chi\right\rfloor$
using the geometric series estimate gives
\begin{align*}
 \chi\left(x\right)+\sum_{k=1}^{\phi\left(x\right)-1}\left(\chi\left(x\right)\cdot 2^{-k}-1\right)
 &\geq \sum_{k=0}^{\left\lfloor\log_2 \chi\right\rfloor}\chi\cdot 2^{-k}-\left(\left\lfloor\log_2 \chi\right\rfloor-1\right)
 \geq 2\chi-\left\lfloor\log_2 \chi\right\rfloor+1.
\end{align*}
Applying this on \eqref{eq: eta beta estim} and using the geometric series formula also for the righthand side of \eqref{eq: eta beta estim} gives
\begin{align}
2\chi-\left\lfloor\log_2 \chi\right\rfloor+1\leq \eta\leq 2\chi.\label{eq: chi eta estim 2}
\end{align}
Given this estimate and assuming that $x\in J_{j,i}^m$ yields
\begin{align}
 \chi\left(x\right)
 \leq \left\lfloor \left(\frac{1}{2^j}-\frac{i+1}{2^{j+m}}\right)^{-1}\right\rfloor
 =\left\lfloor \cfrac{2^{j+m}}{2^m-i-1}\right\rfloor
 \leq \frac{2^{j+m}}{2^m-i-1}.\label{eq: chi leq}
\end{align}
Applying this on the second inequality of \eqref{eq: chi eta estim 2} gives the second estimate of \eqref{eq: rel w eta wo ell}.

On the other hand, if $x\in J_{j,i}^m$ we have that 
\begin{align}
 \chi\left(x\right)
 \geq \left\lfloor \left(\frac{1}{2^j}-\frac{i}{2^{j+m}}\right)^{-1}\right\rfloor
 =\left\lfloor \frac{2^{j+m}}{2^m-i}\right\rfloor
 \geq \frac{2^{j+m}}{2^m-i}-1\label{eq: chi geq 1}
\end{align}
and using \eqref{eq: chi leq} gives 
\begin{align}
 \left\lfloor\log_2 \chi\right\rfloor
 \leq \max_{i\in\{0,\ldots, 2^{m-1}-1\}}\left\lfloor\log_2 \frac{2^{j+m}}{2^m-i-1}\right\rfloor
 \leq \log_2 2^{j+1}=j+1.\label{eq: chi geq 2}
\end{align}
Applying \eqref{eq: chi geq 1} and \eqref{eq: chi geq 2} on \eqref{eq: chi eta estim 2} 
gives the first estimate in \eqref{eq: rel w eta wo ell}.

In order to investigate 
$\eta^{r}(x)$ further, we notice that 
on the one hand
\begin{align}
 \{\chi\leq r\}
 &=\left\{x\colon \left\lfloor\frac{1}{x}\right\rfloor\leq r\right\}
 =\left\{x\colon \log_2\left\lfloor\frac{1}{x}\right\rfloor\leq \log_2 r\right\}\notag\\
 &\supset \left\{x\colon \log_2\frac{1}{x}\leq \left\lfloor\log_2 r\right\rfloor\right\}
 =\left\{x\colon x\geq 2^{-\left\lfloor\log_2 r\right\rfloor}\right\}.\label{eq: chi leq r supset}
\end{align}

On the other hand
\begin{align}
 \{\chi\leq r\}
 &=\left\{x\colon \log_2\left\lfloor\frac{1}{x}\right\rfloor\leq \log_2 r\right\}
 \subset \left\{x\colon \left\lfloor\log_2\left\lfloor\frac{1}{x}\right\rfloor\right\rfloor\leq \left\lfloor\log_2 r\right\rfloor\right\}
 =\left\{x\colon \left\lfloor\log_2\frac{1}{x}\right\rfloor\leq \left\lfloor\log_2 r\right\rfloor\right\}\notag\\
 &\subset \left\{x\colon \log_2\frac{1}{x}\leq \left\lfloor\log_2 r\right\rfloor+1\right\}
 =\left\{x\colon x\geq 2^{-\left\lfloor\log_2 r\right\rfloor-1}\right\}.\label{eq: chi leq r subset}
\end{align}

Furthermore, using the definition 
\[
\eta^{r}(x) 
=\sum_{k=0}^{\phi(x)-1}\chi^r\circ\tau^k(x)
= \sum_{k=0}^{\phi(x)-1}\left(\chi\cdot\mathbbm{1}_{\{\chi\leq r\}}\right)\circ\tau^k(x) 
\]
and applying \eqref{eq: chi leq r supset} and \eqref{eq: chi leq r subset} gives
\begin{align}
\sum_{k=0}^{\phi(x)-1}\left(\chi\cdot\mathbbm{1}_{[2^{-\lfloor\log_2 r\rfloor},1)}\right)\circ\tau^k(x)
&\leq \eta^{r}\left(x\right)
\leq \sum_{k=0}^{\phi(x)-1}\left(\chi\cdot\mathbbm{1}_{[2^{-\lfloor\log_2 r\rfloor-1},1)}\right)\circ\tau^k(x).\label{eq: sum eta ell sum}
\end{align}

Moreover, if $\phi(x)=n$, then $x\in J_{n-1}$
and for all $j\in[0,n-1]\cap\N_0$ we have that 
$\tau^j x=2^j x\in J_{n-j-1}$.
On the other hand, for $q\in\N_0$, we have that
$\{x\geq 2^{-q}\}\cap \Omega'=\bigcup_{j=0}^q J_j\cap\Omega'$.
This implies for all $x\in\Omega'$
\begin{align*}
\sum_{k=0}^{\phi(x)-1}\left(\chi\cdot\mathbbm{1}_{[2^{-q},1)}\right)\circ \tau^k(x)
 &=\sum_{k=0}^{\phi(x)-1}\left(\chi\cdot\mathbbm{1}_{[2^{-q},1)}\right)(2^kx)\\
 &=\sum_{k=\max\{0, \phi(x)-q\}}^{\phi(x)-1}\left(\chi\cdot\mathbbm{1}_{[2^{-q},1)}\right)\left(2^kx\right)\\
 &=\sum_{k=\max\{0, \phi(x)-q\}}^{\phi(x)-1}\chi\left(2^kx\right)\\
 &=\sum_{k=0}^{\phi(x)-\max\{0, \phi(x)-q\}-1}\chi\left(2^{k+\max\{0, \phi(x)-q\}}x\right)\\
 &=\sum_{k=0}^{\phi(x)-\max\{0, \phi(x)-q\}-1}\chi\circ\tau^k\left(2^{\max\{0, \phi(x)-q\}}x\right).
\end{align*}
Moreover,
\[
 \phi\left(2^{\max\{0, \phi(x)-q\}}x\right)=\phi(x)-\max\{0, \phi(x)-q\},
\]
and using the definition of $\eta$ implies 
\begin{align*}
\sum_{k=0}^{\phi(x)-1}\left(\chi\cdot \mathbbm{1}_{[2^{-q},1)}\right)(\tau^k(x))
 &=\sum_{k=\max\{0, \phi(x)-q\}}^{\phi(x)-1}\chi(\tau^k(x))
 =\sum_{k=0}^{\phi(\tau^{\max\{0, \phi(x)-q\}})-1}\chi(\tau^k(x))\\
 &=\eta\left(2^{\max\{0, \phi(x)-q\}}x\right).
\end{align*}
Applying \eqref{eq: sum eta ell sum} gives for all $x\in\Omega'$
\begin{align*}
 \eta\left(2^{\max\{0, \phi(x)-\lfloor\log_2 r\rfloor\}}x\right)
 \leq \eta^r(x)
 \leq \eta\left(2^{\max\{0, \phi(x)-\lfloor\log_2 r\rfloor-1\}}x\right).
\end{align*}

Furthermore, if $x\in J_{j,i}$, then, for all $q\in[0,j]\cap\N_0$, it holds that  $2^{\max\{j-q,0\}}x\in J_{\max\{j,\lfloor\log_2r\rfloor\},i}$.
For $x\in\Omega'$, applying the first/second inequality of \eqref{eq: rel w eta wo ell} on \eqref{eq: sum eta ell sum} and \eqref{eq: def wmr vmr}
gives the first/second inequality of \eqref{eq: rel w eta}.
\end{proof}

\begin{lemma}\label{lem: EW vm}
For all $\epsilon>0$ there exist $M\in\N$, $R>0$ such that for all
$m>\N_{\geq M}$ and $ r\geq R$
\begin{align*}
 \mathbb{E}\left(v_m^{r}\right)
 \leq \left(2+\epsilon\right)\cdot \log r.
\end{align*}
\end{lemma}
\begin{proof}
Since $v_m^{r}$ is a piecewise constant function
attaining the value $y_{\min\{j,\lfloor\log_2 r\rfloor\},i}$ on the interval $J_{j,i}^m$
we have that 
 \begin{align}
  \mathbb{E}\left(v_m^{r}\right)
  &=\sum_{i=0}^{2^{m-1}-1}\sum_{j=0}^{\infty}\lambda\left(J_{j,i}^m\right)\cdot y_{\min\{j,\lfloor\log_2 r\rfloor\},i}\notag\\
  &=\sum_{i=0}^{2^{m-1}-1}\left(\sum_{j=0}^{\lfloor\log_2 r\rfloor}\lambda\left(J_{j,i}^m\right)\cdot y_{j,i}
  +\sum_{j=\lfloor\log_2 r\rfloor+1}^{\infty}\lambda\left(J_{j,i}^m\right)\cdot y_{\lfloor\log_2 r\rfloor,i}\right).\label{eq: E vm estim2}
 \end{align}
From the definition of $J_{j,i}^m$ it follows that
\begin{align}
 \lambda\left(J_{j,i}^m\right)
 =\frac{1}{2^j}-\frac{i}{2^{j+m}}-\left(\frac{1}{2^j}-\frac{i+1}{2^{j+m}}\right)
 =\frac{1}{2^{j+m}}.\label{eq: lambda J}
\end{align}
Hence, inserting this value and the value of $y_{j,i}$ given in \eqref{eq: def y z}
into \eqref{eq: E vm estim2} yields
\begin{align}
 \mathbb{E}\left(v_m^{r}\right)
 &=\sum_{i=0}^{2^{m-1}-1}\left(\sum_{j=0}^{\lfloor\log_2 r\rfloor}\frac{2}{2^m-i-1}
  +\sum_{j=\lfloor\log_2 r\rfloor+1}^{\infty}\frac{1}{2^m-i-1}\cdot \frac{1}{2^{j-\lfloor\log_2 r\rfloor-1}}\right)\notag\\
 &=\sum_{i=0}^{2^{m-1}-1}\left(\frac{2\left(\lfloor\log_2 r\rfloor+1\right)}{2^m-i-1}
  +\frac{2}{2^m-i-1}\right)\notag\\
 &=2\left(\lfloor\log_2 r\rfloor+2\right)\sum_{i=0}^{2^{m-1}-1}\frac{1}{2^m-i-1}.\label{eq: E vm estim}
\end{align}
Since $1/x$ is monotonically decreasing, the last sum can be estimated by the integral
\begin{align*}
\sum_{i=0}^{2^{m-1}-1}\frac{1}{2^m-i-1}
&\leq \int_0^{2^{m-1}}\frac{1}{2^m-x-1}\mathrm{d}x
=\log(2^{m}-1)-\log(2^{m-1}-1)
\end{align*}
and for each $\epsilon>0$ we can find $M$ such that for $m\geq M$
\begin{align*}
\sum_{i=0}^{2^{m-1}-1}\frac{1}{2^m-i-1}
&\leq \left(1+\epsilon/3\right)\cdot \log 2.
\end{align*}
Combining this with \eqref{eq: E vm estim} gives the statement of the lemma.
\end{proof}

The next lemma is the analogous statement to the previous lemma for $w_m^r$.
\begin{lemma}\label{lem: EW wm}
For all $\epsilon>0$ there exist $M\in\N$, $R>0$ such that for all
$m>\N_{\geq M}$ and $ r\geq R$
\begin{align*}
 \mathbb{E}\left(w_m^{r}\right)\geq (2+\epsilon)\cdot\log  r.
\end{align*}
\end{lemma}
\begin{proof}
Analogously as in the proof of Lemma \ref{lem: EW vm} we have that
 \begin{align}
  \mathbb{E}\left(w_m^{r}\right)
  &=\sum_{i=0}^{2^{m-1}-1}\left(\sum_{j=0}^{\lfloor\log_2 r\rfloor-1}\lambda\left(J_{j,i}^m\right)\cdot z_{j,i}
  +\sum_{j=\lfloor\log_2 r\rfloor}^{\infty}\lambda\left(J_{j,i}^m\right)\cdot z_{\lfloor\log_2 r\rfloor,i}\right).\label{eq: E wm estim1}
 \end{align}
From \eqref{eq: lambda J} and \eqref{eq: def y z} it follows that 
\begin{align}
 \MoveEqLeft\sum_{j=0}^{\lfloor\log_2 r\rfloor-1}\lambda\left(J_{j,i}^m\right)\cdot z_{j,i}
  +\sum_{j=\lfloor\log_2 r\rfloor}^{\infty}\lambda\left(J_{j,i}^m\right)\cdot z_{\lfloor\log_2 r\rfloor,i}\notag\\
 &=\sum_{j=0}^{\lfloor\log_2 r\rfloor-1}\frac{2}{2^m-i}
  +\sum_{j=\lfloor\log_2 r\rfloor}^{\infty}\frac{1}{2^m-i}\cdot \frac{1}{2^{j-\lfloor\log_2 r\rfloor-1}}
  -\sum_{j=0}^{\infty}\frac{j+1}{2^{j+m}}\notag\\
 &>\sum_{j=0}^{\lfloor\log_2 r\rfloor-1}\frac{2}{2^m-i}
  -\sum_{j=0}^{\infty}\frac{j+1}{2^{j+m}}\notag\\
 &\geq \frac{2\lfloor\log_2 r\rfloor}{2^m-i}-4,\label{eq: E wm estim2}
\end{align}
for $m$ sufficiently large, where the last inequality follows from 
\[
 \sum_{j=0}^{\infty}\frac{j+1}{2^{j+m}}
 \leq \sum_{j=0}^{\infty}\frac{1}{2^{j/2}}
 =\frac{1}{1-2^{-1/2}}
 <4,
\]
for $m$ sufficiently large. 
Combining \eqref{eq: E wm estim1} and \eqref{eq: E wm estim2} yields
\begin{align}
 \mathbb{E}\left(w_m^{r}\right)
  &\geq \sum_{i=0}^{2^{m-1}-1}\left(\frac{2\lfloor\log_2 r\rfloor}{2^m-i}-4\right)
  =2\lfloor\log_2 r\rfloor\cdot \sum_{i=0}^{2^{m-1}-1}\left(\frac{1}{2^m-i}\right)-2^{m+1},\label{eq: E wm estim3}
\end{align}
for $m$ sufficiently large, where the sum can be estimated by an integral as follows
\begin{align*}
\sum_{i=0}^{2^{m-1}-1}\frac{1}{2^m-i}
&\geq \int_{-1}^{2^{m-1}-1}\frac{1}{2^m-x}\mathrm{d}x
=\log(2^{m}+1)-\log(2^{m-1}+1)
\geq (1-\epsilon/3)\log 2,
\end{align*}
for $m$ sufficiently large.
Combining this with \eqref{eq: E wm estim3} gives 
\begin{align*}
 \mathbb{E}\left(w_m^{r}\right)
 &\geq (1-\epsilon/3)\log 2\cdot 2\left\lfloor\log_2\right\rfloor-2^{m+1}
 \geq (2-\epsilon)\cdot \log_2 r,
\end{align*}
for $ r$ and $m$ sufficiently large (if $\epsilon$ was chosen sufficiently small)
and thus the statement of the lemma follows.
\end{proof}

In the following lemma we give a statement related to Lemma \ref{lem: Tn tn a.s.}
using $\left(\eta\circ \tau_B^{n-1}\right)$ instead of $\left(\chi\circ\tau^{n-1}\right)$. 
With the previously attained properties of $\eta$ we are able to prove this lemma.

To formulate this lemma let, for any function $\varphi\colon [0,1)\to\mathbb{R}_{>0}$ and $r>0$,
\[
 \mathsf{T}_n^r\varphi=\sum_{k=1}^n\varphi\circ \tau_B^{k-1}\cdot\mathbbm{1}_{\{\varphi\circ \tau_B^{k-1}\leq r\}}
\]
and as in Lemma \ref{lem: Tn tn a.s.} we set $t_n=n\cdot \left(\log\left(n\right)\right)^{3/4}$.
\begin{lemma}\label{lem: Sn* allg}
We have that
\begin{align*}
\lim_{n\to\infty}\frac{\mathsf{T}_n^{2t_{n}} \eta}{\mathbb{E}\left(\mathsf{T}_n^{2t_{n}} \eta\right)}=1~\text{ a.s.}
\end{align*}
\end{lemma}
In order to prove this lemma we first need an exponential inequality. 
The following lemma generalizes \emph{Bernstein's
inequality} and can be found for example in \cite{hoeffding_probability_1963}.
\begin{lemma}[Generalized Bernstein inequality]\label{Bernstein} For $n\in\mathbb{N}$
let $Y_{1},\ldots,Y_{n}$ be independent random variables such that
$\left\Vert Y_{i}-\mathbb{E}\left(Y_{i}\right)\right\Vert _{\infty}\leq M<\infty$
for $i=1,\ldots,n$. Let $Z_{n}=\sum_{i=1}^{n}Y_{i}$. Then
we have for all $t>0$ that
\begin{align*}
\mathbb{P}\left(\max_{k\leq n}\left|Z_{k}-\mathbb{E}\left(Z_{k}\right)\right|\geq t\right)\leq2\exp\left(-\frac{t^{2}}{2\mathbb{V}\left(Z_{n}\right)+\frac{2}{3}Mt}\right).
\end{align*}
\end{lemma}
With the help of Lemma \ref{Bernstein} we are able to prove the following
Lemma \ref{Bernstein 1} for the special case of non-negative random
variables.
\begin{lemma}
\label{Bernstein 1} For $n\in\mathbb{N}$ let $Y_{1},\ldots,Y_{n}$
be i.i.d.\ non-negative random variables such that $Y_{1}\leq K<\infty$.
Let $Z_{n}=\sum_{i=1}^{n}Y_{i}$. Then we have for all $\kappa>0$
that
\begin{align*}
\mathbb{P}\left(\max_{k\leq n}\left|Z_{k}-\mathbb{E}\left(Z_{k}\right)\right|\geq\kappa\cdot\mathbb{E}\left(Z_{n}\right)\right)\leq2\exp\left(-\frac{3\kappa^{2}}{6+2\kappa}\cdot\frac{\mathbb{E}\left(Z_{n}\right)}{K}\right).
\end{align*}
\end{lemma}
\begin{proof} 
First note that we may chose $M=K$ in Lemma \ref{Bernstein}
to obtain
\[
\max_{1\leq i\leq n}\left|Y_{i}-\mathbb{E}\left(Y_{i}\right)\right|\leq K=M.
\]
Since
\begin{align*}
\mathbb{V}\left(Z_{n}\right) & = n\left(\int_{0}^{K}x^{2}\mathrm{d}F\left(x\right)-\left(\int_{0}^{K}x\mathrm{d}F\left(x\right)\right)^{2}\right)<n\int_{0}^{K}x^{2}\mathrm{d}F\left(x\right)\\
 & <  n\cdot K\int_{0}^{K}x\mathrm{d}F\left(x\right)=K\cdot\mathbb{E}\left(Z_{n}\right),
\end{align*}
it follows by Lemma \ref{Bernstein} that
\begin{align*}
\mathbb{P}\left(\max_{k\leq n}\left|Z_{k}-\mathbb{E}\left(Z_{k}\right)\right|\geq\kappa\cdot\mathbb{E}\left(Z_{n}\right)\right) & \leq2\exp\left(-\frac{\kappa^{2}\cdot\mathbb{E}\left(Z_{n}\right)^{2}}{2\mathbb{V}\left(Z_{n}\right)+\frac{2}{3}\kappa\cdot K\cdot\mathbb{E}\left(Z_{n}\right)}\right)\\
 & <2\exp\left(-\frac{\kappa^{2}}{2+\frac{2}{3}\kappa}\cdot\frac{\mathbb{E}\left(Z_{n}\right)}{K}\right)\\
 & =2\exp\left(-\frac{3\kappa^{2}}{6+2\kappa}\cdot\frac{\mathbb{E}\left(Z_{n}\right)}{K}\right).
\end{align*}
\end{proof}

Now we are able to start with the proof of Lemma \ref{lem: Sn* allg}.
\begin{proof}[Proof of Lemma \ref{lem: Sn* allg}]
The proof can be summarized into two main steps. 
First we fix $\epsilon>0$ and find sets $(A_i)_{i\in I}$ and a corresponding index set $I$ such that there exists $M\in\N$ fulfilling
\begin{align*}
\MoveEqLeft\bigcup_{n\in\mathbb{N}_{\geq M}}\left\{\mathsf{T}_n^{2t_n}\eta-\mathbb{E}\left(\mathsf{T}_n^{2t_n}\eta\right)>\epsilon\cdot \mathbb{E}\left(\mathsf{T}_n^{2t_n}\eta\right)\right\}\notag\\
&\cup\bigcup_{n\in\mathbb{N}_{\geq M}}\left\{\mathbb{E}\left(\mathsf{T}_n^{2t_n}\eta\right)-\mathsf{T}_n^{2t_n}\eta>\epsilon\cdot \mathbb{E}\left(\mathsf{T}_n^{2t_n}\eta\right)\right\}
\subset \bigcup_{i\in I}A_i.
\end{align*}
Afterwards we calculate $\lambda\left(A_i\right)$ for each $i\in I$ and show that 
$\sum_{i\in I}\lambda\left(A_i\right)<\infty$. Applying then the first Borel-Cantelli lemma gives 
the statement of the lemma.

We start by introducing the following notation:
\begin{align*}
 \overline{v}_m^{r}&\coloneqq v_m^{r}-\int v_m^{r}\mathrm{d}\lambda\,\,\,\,\,\,\text{ and }\,\,\,\,\,\,
 \overline{w}_m^{r}\coloneqq w_m^{r}-\int w_m^{r}\mathrm{d}\lambda.
\end{align*}

Using the second inequality of Lemma \ref{lem: rel w eta ell} we first obtain the following inclusion 
\begin{align}
 \MoveEqLeft\left\{\mathsf{T}_n^{2t_n}\eta-\mathbb{E}\left(\mathsf{T}_n^{2t_n}\eta\right)>\epsilon\cdot \mathbb{E}\left(\mathsf{T}_n^{2t_n}\eta\right)\right\}\notag\\
 &\subset \left\{\mathsf{T}_n^{2t_n}v_m-\mathbb{E}\left(\mathsf{T}_n^{2t_n}\eta\right)>\epsilon\cdot \mathbb{E}\left(\mathsf{T}_n^{2t_n}\eta\right)\right\}\notag\\
 &\subset \left\{\mathsf{T}_n^{2t_n}\overline{v}_m>\epsilon\cdot \mathbb{E}\left(\mathsf{T}_n^{2t_n}v_m\right)-\left(1+\epsilon\right)\cdot \left(\mathbb{E}\left(\mathsf{T}_n^{2t_n}v_m\right)-\mathbb{E}\left(\mathsf{T}_n^{2t_n}\eta\right)\right)\right\}.\label{eq: Teta - ETeta}
\end{align}
Combining Lemma \ref{lem: rel w eta ell}, Lemma \ref{lem: EW vm}, and Lemma \ref{lem: EW wm} 
yields that for all $\widetilde{\epsilon}>0$ there exist $M,N\in\N$ such that, for all $m\geq M$ and $n\geq N$,
\begin{align}
 \mathbb{E}\left(\mathsf{T}_n^{2t_n}v_m\right)-\mathbb{E}\left(\mathsf{T}_n^{2t_n}\eta\right)
 &\leq \mathbb{E}\left(\mathsf{T}_n^{2t_n}v_m\right)-\mathbb{E}\left(\mathsf{T}_n^{2t_n}w_m\right)
 \leq 2\widetilde{\epsilon}\cdot\mathbb{E}\left(\mathsf{T}_n^{2t_n}v_m\right).\label{eq: Ev-Ew}
\end{align}
Setting $\widetilde{\epsilon}=\epsilon/6$ yields, for $\epsilon\in(0,1)$, 
\begin{align*}
 \epsilon\cdot \mathbb{E}\left(\mathsf{T}_n^{2t_n}v_m\right)-\left(1+\epsilon\right)\cdot \left(\mathbb{E}\left(\mathsf{T}_n^{2t_n}v_m\right)-\mathbb{E}\left(\mathsf{T}_n^{2t_n}\eta\right)\right)
 &\geq \left(\epsilon-\left(1+\epsilon\right)\cdot 2\cdot \frac{\epsilon}{6}\right)\cdot \mathbb{E}\left(\mathsf{T}_n^{2t_n}v_m\right)\\
 &\geq \frac{\epsilon}{3}\cdot \mathbb{E}\left(\mathsf{T}_n^{2t_n}v_m\right).
\end{align*}
Inserting this in \eqref{eq: Teta - ETeta} yields
\begin{align}
\left\{\mathsf{T}_n^{2t_n}\eta-\mathbb{E}\left(\mathsf{T}_n^{2t_n}\eta\right)>\epsilon\cdot \mathbb{E}\left(\mathsf{T}_n^{2t_n}\eta\right)\right\}
&\subset\left\{\left|\mathsf{T}_n^{2t_n}\overline{v}_m\right|>\epsilon/3\cdot \mathbb{E}\left(\mathsf{T}_n^{2t_n}v_m\right)\right\},\label{eq: trunc sum estim 1a}
\end{align}
for $m,n$ sufficiently large.
Analogously, we obtain by the first inequality of Lemma \ref{lem: rel w eta ell}
\begin{align}
\MoveEqLeft\left\{\mathbb{E}\left(\mathsf{T}_n^{2t_n}\eta\right)-\mathsf{T}_n^{2t_n}\eta>\epsilon\cdot \mathbb{E}\left(\mathsf{T}_n^{2t_n}\eta\right)\right\}\notag\\
&\subset\left\{\mathbb{E}\left(\mathsf{T}_n^{2t_n}\eta\right)-\mathsf{T}_n^{2t_n}w_m>\epsilon\cdot \mathbb{E}\left(\mathsf{T}_n^{2t_n}w_m\right)\right\}\notag\\
&\subset\left\{\mathbb{E}\left(\mathsf{T}_n^{2t_n}w_m\right)-\mathsf{T}_n^{2t_n}w_m>\epsilon\cdot \mathbb{E}\left(\mathsf{T}_n^{2t_n}w_m\right)-\left(\mathbb{E}\left(\mathsf{T}_n^{2t_n}\eta\right)-\mathbb{E}\left(\mathsf{T}_n^{2t_n}w_m\right)\right)\right\}\notag\\
&\subset \left\{\left|\mathsf{T}_n^{2t_n}w_m-\mathbb{E}\left(\mathsf{T}_n^{2t_n}w_m\right)\right|>\epsilon/3\cdot \mathbb{E}\left(\mathsf{T}_n^{2t_n}w_m\right)\right\}\notag\\
&= \left\{\left|\mathsf{T}_n^{2t_n}\overline{w}_m\right|>\epsilon/3\cdot \mathbb{E}\left(\mathsf{T}_n^{2t_n}w_m\right)\right\},\label{eq: trunc sum estim 1b}
\end{align}
sufficiently large,
where the forth line follows from a similar calculation as in \eqref{eq: Ev-Ew}.

For the following let us always assume that $m$ is large enough that the above inclusions hold.
We first proceed with the estimation of \eqref{eq: trunc sum estim 1a}.
The estimation of \eqref{eq: trunc sum estim 1b} follows very much analogously, as we will see later on.

We define 
$I_{j}^m=[m^{j},m^{j+1}-1]\cap \mathbb{N}$ for $j,m\in\mathbb{N}$
(a generalisation of $I_j=I_j^2$ defined in \eqref{eq: Kk})
and for $n\in I_{j}^m$ we have
\begin{align*}
 \mathsf{T}_n^{2t_n}\overline{v}_m
 &=\sum_{u=0}^{m-1}\sum_{\ell=0}^{\gamma\left(n,u\right)}\overline{v}_m^{2t_n}\circ \tau_B^{\ell\cdot m+u},
\end{align*}
where $\gamma\left(n,u\right)$ can be uniquely determined and takes values in the interval $I_{j-1}^m$
if $n\in I_j^m$. 
However, given our following estimations there is no need to further investigate which exact value $\gamma\left(n,u\right)$ attains for given $n$ and $u$.
With this in mind we obtain the following inclusion
\begin{align}
 \left\{\left|\mathsf{T}_n^{2t_n}\overline{v}_m\right|>\epsilon/3\cdot \mathbb{E}\left(\mathsf{T}_n^{2t_n}v_m\right)\right\}
 &\subset \bigcup_{u=0}^{m-1} \left\{\left|\sum_{\ell=0}^{\gamma\left(n,u\right)}\overline{v}_m^{2t_n}\circ \tau_B^{\ell\cdot m+u}\right|>\frac{\epsilon}{3\cdot m}\cdot \mathbb{E}\left(\mathsf{T}_n^{2t_n}v_m\right)\right\}.\label{eq: trunc sum estim 2}
\end{align}
The reason we make this last estimate is that, by Lemma \ref{lem: indep of mth entry}, $\sum_{\ell=0}^{\gamma\left(n,u\right)}v_m^{2t_n}\circ \tau_B^{\ell\cdot m+u}$
is a sum of independent random variables. This will later facilitate to estimate the probability of the single events.

In the next steps we aim to combine all events for $n\in I_j^m$.
For doing so we notice that for $n\in I_{j}^m$ we have that
\begin{align*}
 \mathbb{E}\left(\mathsf{T}_n^{2t_n}v_m\right)
 &=n\cdot \mathbb{E}\left(v_m^{2t_n}\right)
 \geq m^j\cdot \mathbb{E}\left(v_m^{2t_n}\right)
 =\frac{1}{m}\cdot \mathbb{E}\left(\sum_{\ell=0}^{m^{j-1}-1}v_m^{2t_n}\circ \tau_B^{\ell\cdot m+u}\right).
\end{align*}
For the following we set $\epsilon_1=\epsilon/(3m)$.
Hence, $n\in I_{j}^m$ and thus $\gamma\left(n,u\right)\in I_{j-1}^m$ implies
\begin{align}
\MoveEqLeft\left\{\left|\sum_{\ell=0}^{\gamma\left(n,u\right)}\overline{v}_m^{2t_n}\circ \tau_B^{\ell\cdot m+u}\right|>\epsilon_1\cdot \mathbb{E}\left(\mathsf{T}_n^{2t_n}v_m\right)\right\}\notag\\
&\subset \left\{ \max_{k\in I_{j-1}^m}\left|\sum_{\ell=0}^{k}\overline{v}_m^{2t_n}\circ \tau_B^{\ell\cdot m+u}\right|>\epsilon_1\cdot \mathbb{E}\left(\sum_{\ell=0}^{m^{j-1}-1}v_m^{2t_n}\circ \tau_B^{\ell\cdot m+u}\right)\right\}.\label{eq: trunc sum estim 3}
\end{align}
In order to keep our notation short we introduce the set $\Gamma$
for an index set $J$, a non-negative integrable observable $\varphi$, and a transformation $\xi$ as
\begin{align*}
 \Gamma\left(J,\varphi, \xi\right)
 =\left\{ \max_{k\in J}\left|\sum_{\ell=0}^{k}\left(\varphi-\mathbb{E}\left(\varphi\right)\right)\circ \xi^{\ell}\right|
 >\epsilon_1\cdot \mathbb{E}\left(\sum_{\ell=0}^{\min\left\{k\in J\right\}-1}\varphi\circ \xi^{\ell}\right)\right\}.
\end{align*}
Then the righthand side of \eqref{eq: trunc sum estim 3}
writes as 
$\Gamma\left(I_{j-1}^m, v_m^{2t_n}\circ \tau_B^u, \tau_B^{m}\right)$.

For different $n\in I_j^m$ the term $2t_n$ takes different values. 
In the next steps we aim to tackle this problem in order to obtain a concise expression 
of the righthand side of \eqref{eq: trunc sum estim 3}
which only depends on $m$ and $j$.
Since $2t_n$ is monotonically increasing, we have that 
$2t_n\in\left[2t_{m^j}, 2t_{m^{j+1}}\right)$
if $n\in I_{j}^m$.

We set 
\begin{align*}
 r_j=\lfloor\log_2(2t_{m^j})\rfloor\,\,\,\text{ and }\,\,\,s_j=\lfloor\log_2(2t_{m^{j+1}})\rfloor.
\end{align*}
Note that there is a dependence on $m$ in $r_j$ and $s_j$ which we omit for brevity.
Keeping in mind that by its definition 
$v_m^{r}=v_m^{k}$ if $\lfloor\log_2r\rfloor=\lfloor\log_2k\rfloor$ and using the above
notation we obtain from \eqref{eq: trunc sum estim 3}
\begin{align}
\bigcup_{n\in I_j^m}\left\{\left|\sum_{\ell=0}^{\gamma\left(n,u\right)}\overline{v}_m^{2t_n}\circ \tau_B^{\ell\cdot m+u}\right|>\epsilon_1\cdot \mathbb{E}\left(\mathsf{T}_n^{2t_n}v_m\right)\right\}
&\subset \bigcup_{p =r_j}^{s_j} \Gamma\left(I_{j-1}^m, v_m^{2^{p}}\circ \tau_B^u, \tau_B^{m}\right).\label{eq: trunc sum estim 4} 
\end{align}

The reason we do this estimate is that in stead of considering $\# I_{j-1}^m=m^j-1-m^{j-1}$ summands
we later only consider $\#[r_j,s_j]\cap\N$ summands (estimated in \eqref{eq: sj-rj})
in the Borel-Cantelli sum yielding a better result.

Combining \eqref{eq: trunc sum estim 1a}
with \eqref{eq: trunc sum estim 2}, \eqref{eq: trunc sum estim 3}, and \eqref{eq: trunc sum estim 4}
gives the existence of $N,J\in \N$ such that
\begin{align}
 \bigcup_{n\in\mathbb{N}_{\geq N}}\left\{\mathsf{T}_n^{2t_n}\eta-\mathbb{E}\left(\mathsf{T}_n^{2t_n}\eta\right)>\epsilon\cdot \mathbb{E}\left(\mathsf{T}_n^{2t_n}\eta\right)\right\}
 &\subset \bigcup_{j\in\mathbb{N}_{\geq J}}\bigcup_{u=0}^{m}\bigcup_{p=r_j}^{s_j} \Gamma\left(I_{j-1}^m, v_m^{2^{p}}\circ \tau_B^u, \tau_B^{m}\right).\label{eq: T subset Gamma1}
\end{align}

The case starting with \eqref{eq: trunc sum estim 1b} can be done analogously resulting in  the existence of $N,J\in \N$ such that
\begin{align}
 \bigcup_{n\in\mathbb{N}_{\geq N}}\left\{\mathbb{E}\left(\mathsf{T}_n^{2t_n}\eta\right)-\mathsf{T}_n^{2t_n}\eta>\epsilon\cdot \mathbb{E}\left(\mathsf{T}_n^{2t_n}\eta\right)\right\}
 &\subset \bigcup_{j\in\mathbb{N}_{\geq J}}\bigcup_{u=0}^{m}\bigcup_{p=r_j}^{s_j} \Gamma\left(I_{j-1}^m, w_m^{2^{p}}\circ \tau_B^u, \tau_B^{m}\right).\label{eq: T subset Gamma2}
\end{align}

We now start the second part of the proof by estimating 
\begin{align*}
  \sum_{j=J}^{\infty}\sum_{u=0}^{m-1}\sum_{p=r_j}^{s_j}\left(\lambda\left( \Gamma\left(I_{j-1}^m, v_m^{2^{p}}\circ \tau_B^u, \tau_B^{m}\right)\right)
 +\lambda\left( \Gamma\left(I_{j-1}^m, w_m^{z_{p,q}}\circ \tau_B^u, \tau_B^{m}\right)\right)\right).
\end{align*}

As mentioned earlier in this section,
by Lemma \ref{lem: indep of mth entry},
$\sum_{\ell=0}^{k}\overline{v}_m^{2^p}\circ \tau_B^{\ell\cdot m+u}$
is a sum of independent
random variables and we can apply Lemma \ref{Bernstein 1}.
We note
\begin{align}
 v_m^{2^{p}}
 &\leq\max_{i\in\{0,\ldots, 2^{m-1}-1\}}y_{\lfloor\log_2 2^p\rfloor, i}
 =\max_{i\in\{0,\ldots, 2^{m-1}-1\}}y_{p, i}\notag\\
 &=\max_{i\in\{0,\ldots, 2^{m-1}-1\}}\frac{2^{p+m+1}}{2^m-i-1}
 =\frac{2^{p+m+1}}{2^m-(2^{m-1}-1)-1}
 =2^{p+2}.\label{eq: vm2p}
\end{align}
This yields, for all $u\in\N_0$,
\begin{align}
\MoveEqLeft\lambda\left(\Gamma\left(I_{j-1}^m, v_m^{2^p}\circ \tau_B^u, \tau_B^{m}\right)\right)\notag\\
&= \lambda\left(\left\{ \max_{k\in I_{j-1}^m}\left|\sum_{\ell=0}^{k}\overline{v}_m^{2^p}\circ \tau_B^{\ell\cdot m+u}\right|
 >\epsilon_1\cdot \mathbb{E}\left(\sum_{\ell=0}^{\min\left\{k\in I_{j-1}^m\right\}-1}v_m^{2^p}\circ \tau_B^{\ell\cdot m+u}\right)\right\}\right)\notag\\
&\leq \lambda\left(\left\{ \max_{k\leq m^{j-1}-1}\left|\sum_{\ell=0}^{k}\overline{v}_m^{2^p}\circ \tau_B^{\ell\cdot m+u}\right|
 >\epsilon_1\cdot \mathbb{E}\left(\sum_{\ell=0}^{m^j-1}v_m^{2^p}\circ \tau_B^{\ell\cdot m+u}\right)\right\}\right)\notag\\
&\leq 2\exp\left(-\epsilon_2\cdot m^{j-1}\cdot \frac{\mathbb{E}\left(v_m^{2^p}\right)}{2^{p+2}}\right)\label{eq: lambda < exp}
\end{align}
with
$\epsilon_2\coloneqq\left(3\epsilon_1^2\right)/\left(6+2\epsilon_1\right)$.
Further note that by Lemma \ref{lem: rel w eta ell} and \eqref{eq: vm2p} we have $w_m^{2^p}\leq v_m^{2^p}\leq 2^{p+2}$.
Thus, an analogous calculation as above yields
\begin{align}
\lambda\left(\Gamma\left(I_{j-1}^m, w_m^{2^p}\circ \tau_B^u, \tau_B^{m}\right)\right)
&\leq 2\exp\left(-\epsilon_2\cdot m^{j-1}\cdot \frac{\mathbb{E}\left(w_m^{2^p}\right)}{2^{p+2}}\right).\label{eq: lambda < exp1}
\end{align}
Next note that
by Lemma \ref{lem: EW wm},
for $\epsilon\in (0,2-1/\log 2]$, there exists $M,L$ such that, for $m\geq M$ and $2^{p}>L$,
\begin{align}
 \mathbb{E}\left(w_m^{2^p}\right)\geq (2-\epsilon)\log 2^p\geq p.\label{eq: Ewm geq}
\end{align}
Combining \eqref{eq: lambda < exp} and \eqref{eq: lambda < exp1}
with \eqref{eq: Ewm geq} 
and noting that, by Lemma \ref{lem: rel w eta ell}, $\mathbb{E}\left(v_m^{2^p}\right) \geq \mathbb{E}\left(w_m^{2^p}\right)$
yields
\begin{align}
 \lambda\left( \Gamma\left(I_{j-1}^m, v_m^{2^p}\circ \tau_B^u, \tau_B^{m}\right)\right)
 +\lambda\left( \Gamma\left(I_{j-1}^m, w_m^{2^p}\circ \tau_B^u, \tau_B^{m}\right)\right)
 \leq 4\exp\left(-\epsilon_2\cdot m^{j-1}\cdot \frac{p}{2^{p+2}}\right).\label{eq: sj estim}
\end{align}
Since $p/2^{p+2}$ is monotonically decreasing for $p\in [r_j,s_j]$ 
and $j$ sufficiently large, we have for $p\in [r_j,s_j]$
\begin{align}
\frac{p}{2^{p+2}}
&\geq\frac{s_j}{2^{s_j+2}}
 \geq \frac{\log_2 \left(2 t_{m^{j+1}}\right)}{2^{\log_2 \left(2 t_{m^{j+1}}\right)+3}}
 \geq \frac{\log_2 t_{m^{j+1}}}{2^4\cdot t_{m^{j+1}}}
 \geq\frac{(j+1)\log_2 m}{2^4\cdot m^{j+1}\cdot \left(\left(j+1\right)\log m\right)^{3/4}}
 \geq \frac{j^{1/4}}{2^4\cdot m^{j+1}}.\notag
\end{align}

Combining this with \eqref{eq: sj estim} yields, for all $p\in [r_j,s_j]\cap\N$,
\begin{align*}
 \MoveEqLeft \lambda\left( \Gamma\left(I_{j-1}^m, v_m^{2^p}\circ \tau_B^u, \tau_B^{m}\right)\right)
 +\lambda\left( \Gamma\left(I_{j-1}^m, w_m^{2^p}\circ \tau_B^u, \tau_B^{m}\right)\right)\notag\\
 &\leq 4\exp\left(-\epsilon_2\cdot m^{j-1}\cdot \frac{s_j}
 {2^{s_j+2}}\right)
 \leq 4\exp\left(-\epsilon_2\cdot m^{j-1}\cdot \frac{j^{1/4}}
 {2^4\cdot m^{j+1}}\right)\notag\\
 &= 4\exp\left(-\frac{\epsilon_2}{2^4\cdot m^2}\cdot j^{1/4}\right)
 \leq 4\exp\left(-j^{1/5}\right),
\end{align*}
for $j$ sufficiently large.

Furthermore,
\begin{align}
 s_j-r_j
 &=\lfloor\log_2(2t_{m^{j+1}})\rfloor-\lfloor\log_2(2t_{m^j})\rfloor
 \leq \log_2\left(\frac{2t_{m^{j+1}}}{2t_{m^j}}\right)+1\notag\\
 &=\log_2\left(\frac{m^{j+1}\cdot\left(\log m^{j+1}\right)^{3/4}}{m^j\cdot\left(\log m^j\right)^{3/4}}\right)+1
 \leq \log_2m+2,\label{eq: sj-rj}
\end{align}
for $j$ sufficiently large.
This implies 
\begin{align*}
 \MoveEqLeft \sum_{p=r_j}^{s_j}\left(\lambda\left( \Gamma\left(I_{j-1}^m, v_m^{2^p}\circ \tau_B^u, \tau_B^{m}\right)\right)
 +\lambda\left( \Gamma\left(I_{j-1}^m, w_m^{2^p}\circ \tau_B^u, \tau_B^{m}\right)\right)\right)\notag\\
 &\leq 4\sum_{p=r_j}^{s_j}\exp\left(-j^{1/5}\right)
 \leq  4\left(\log_2m+2\right)\exp\left(-j^{1/5}\right),
\end{align*}
for $j$ sufficiently large.

Since $\lambda\left( \Gamma\left(I_{j-1}^m, v_m^{2^p}\circ \tau_B^u, \tau_B^{m}\right)\right)$
and $\lambda\left( \Gamma\left(I_{j-1}^m, w_m^{2^p}\circ \tau_B^u, \tau_B^{m}\right)\right)$ do not differ for different $u\in\N_0$, see \eqref{eq: lambda < exp}, we obtain
\begin{align*}
 \MoveEqLeft \sum_{u=0}^{m-1}\sum_{p=r_j}^{s_j}\left(\lambda\left( \Gamma\left(I_{j-1}^m, v_m^{2^p}\circ \tau_B^u, \tau_B^{m}\right)\right)
 +\lambda\left( \Gamma\left(I_{j-1}^m, w_m^{2^p}\circ \tau_B^u, \tau_B^{m}\right)\right)\right)\notag\\
 &\leq 4m\left(\log_2m+2\right)\exp\left(-j^{1/5}\right),
\end{align*}
for $j$ sufficiently large.
Finally, if we choose $J$ sufficiently large, then
\begin{align*}
\MoveEqLeft\sum_{j=J}^{\infty}\sum_{u=0}^{m-1}\sum_{p=r_j}^{s_j}\left(\lambda\left( \Gamma\left(I_{j-1}^m, v_m^{2^p}\circ \tau_B^u, \tau_B^{m}\right)\right)
 +\lambda\left( \Gamma\left(I_{j-1}^m, w_m^{2^p}\circ \tau_B^u, \tau_B^{m}\right)\right)\right)\\
 &\leq \sum_{j=J}^{\infty}4m\left(\log_2m+2\right)\exp\left(-j^{1/5}\right)<\infty.
\end{align*}
An application of the first Borel-Cantelli lemma on \eqref{eq: T subset Gamma1} and \eqref{eq: T subset Gamma2} gives the statement of the lemma. 
\end{proof}

The next lemma gives a statement about average hitting times and will later give us the possibility to 
compare $\mathsf{T}_n^{2t_n}\eta$ with $T_n^{2t_n}$.

\begin{lemma}\label{lem: sum phi 2n}
We have 
\begin{align*}
 \lambda\left(\left|\sum_{k=0}^{n-1}\phi\circ \tau_B^k-2n\right|>n^{3/4}~\text{ i.o.}\right)=0.
\end{align*}
\end{lemma}
\begin{proof}
First note, that since $\phi$ is measurable on $\mathcal{J}$, by Corollary \ref{cor: indep of 1st entry},
the sequence of random variables $\left(\phi\circ\tau_B^k\right)_{k\in\mathbb{N}}$ is independent.
We define the following sets
\begin{align*}
 \Upsilon_n&= \left\{\left|\sum_{k=0}^{n-1}\phi\circ \tau_B^kx-2n\right|>n^{3/4}\right\}\,\,\,\,\,\text{ and }\,\,\,\,\,
 \Xi_n= \left\{\#\left\{k\leq n\colon \phi\circ \tau_{B}^{k-1}>2\log n\right\}=0\right\}.
\end{align*}
If we denote by $A^c$ the complementary event of an event $A$,
then we have
\begin{align}
\lambda\left(\limsup_{n\to\infty}\Upsilon_n\right)
 &\leq \lambda\left(\limsup_{n\to\infty}\left(\left(\Upsilon_n\cap\Xi_n\right)\cup\Xi_n^c\right)\right)
 =\lambda\left(\limsup_{n\to\infty}\left(\Upsilon_n\cap\Xi_n\right)\cup\limsup_{n\to\infty}\Xi_n^c\right)\notag\\
 &\leq \lambda\left(\limsup_{n\to\infty}\left(\Upsilon_n\cap\Xi_n\right)\right)+\lambda\left(\limsup_{n\to\infty}\Xi_n^c\right).\label{eq: summand number>1}
\end{align}
In order to estimate the first summand of \eqref{eq: summand number>1}
we set $\phi^{r}=\phi\cdot\mathbbm{1}_{\left\{\phi\leq r\right\}}$. 
Then we have
\begin{align}
 \Upsilon_n\cap\Xi_n
 &\subset \left\{\left|\sum_{k=0}^{n-1}\phi^{2\log n}\circ \tau_B^k-2n\right|>n^{3/4}\right\}\notag\\
 &= \left\{\left|\sum_{k=0}^{n-1}\left(\phi^{2\log n}-\mathbb{E}\left(\phi^{2\log n}\right)\right)\circ \tau_B^k-\left(2n-n\cdot \mathbb{E}\left(\phi^{2\log n}\right)\right)\right|>n^{3/4}\right\}\notag\\
 &\subset \left\{\left|\sum_{k=0}^{n-1}\left(\phi^{2\log n}-\mathbb{E}\left(\phi^{2\log n}\right)\right)\circ \tau_B^k\right|>n^{3/4}-\left(2n-n\cdot \mathbb{E}\left(\phi^{2\log n}\right)\right)\right\}.\label{eq: Ups Xi}
\end{align}
We aim to apply Lemma \ref{Bernstein} for which we need to calculate 
$\gamma_n=\left(2-\mathbb{E}\left(\phi^{2\log n}\right)\right)\cdot n$ first.
Obviously, 
\[
 \mathbb{E}\left(\phi^{2\log n}\right)
 = \sum_{k=1}^{\left\lfloor2\log n\right\rfloor}k\cdot \lambda\left(\phi=k\right)
 = \sum_{k=1}^{\left\lfloor2\log n\right\rfloor}k\cdot 2^{-k}.
\]
This can be easily seen, as $\phi(x)=k$ if $x\in J_{k-1}$ and $\lambda\left(J_{k-1}\right)=2^{-k}$.
Remember that $\sum_{k=1}^{\infty}k/2^k=2$. 
This gives
\begin{align*}
2-\mathbb{E}\left(\phi^{2\log n}\right)
 &=2-\sum_{k=1}^{\left\lfloor2\log n\right\rfloor}k\cdot 2^{-k}
 =\sum_{k=\left\lfloor2\log n\right\rfloor+1}^{\infty}k\cdot 2^{-k}.
\end{align*}

Calculating the remainder term with $j=\left\lfloor2\log n\right\rfloor+1$ 
and applying the geometric series formula gives 
\begin{align*}
 2-\mathbb{E}\left(\phi^{2\log n}\right)
 &=\sum_{k=j}^{\infty}\frac{k}{2^k}
 =j\cdot\sum_{k=j}^{\infty}\frac{1}{2^k}+\sum_{k=j}^{\infty}\frac{k-j}{2^k}
 =j\cdot 2^{-j+1}+2^{-j}\sum_{k=0}^{\infty}\frac{k}{2^{k}}\\
 &=j\cdot 2^{-j+1}+2^{-j+1}
 =\frac{\left\lfloor2\log n\right\rfloor+2}{2^{\left\lfloor2\log n\right\rfloor}}.
\end{align*}
This yields 
\begin{align*}
 \gamma_n
 \leq \frac{2\cdot\left(2\log n +2\right)}{n^{2\cdot \log 2}}\cdot n
 =4\cdot\left(\log n +1\right)\cdot n^{1-2\log 2}\leq n^{1/2},
\end{align*}
for $n$ sufficiently large.
Thus,
$n^{3/4}-\gamma_n \geq n^{5/8}$,
for $n$ sufficiently large.
Combining this with \eqref{eq: Ups Xi} and applying Lemma \ref{Bernstein} yields
\begin{align}
\lambda\left(\Upsilon_n\cap\Xi_n\right)
 &\leq \lambda\left(\left|\sum_{k=0}^{n-1}\left(\phi^{2\log n}-\mathbb{E}\left(\phi^{2\log n}\right)\right)\circ \tau_B^k\right|>n^{5/8}\right)\notag\\
 &\leq  \exp\left(-\frac{n^{5/4}}{2\mathbb{V}\left(\sum_{k=0}^{n-1}\phi^{2\log n}\circ \tau_B^k\right)+\frac{2}{3}\cdot n^{5/8}\cdot 2\log n}\right),\label{eq: lambda sum phi bernstein}
\end{align}
for $n$ sufficiently large.
We have, using independence of $\phi\circ \tau_B^k$
and a similar approach as in the calculation of $\mathbb{E}\left(\phi^{2\log n}\right)$,
that 
\begin{align*}
\mathbb{V}\left(\sum_{k=0}^{n-1}\phi^{2\log n}\circ \tau_B^k\right)
&=n\cdot \mathbb{V}\left(\phi^{2\log n}\right)
\leq n\cdot \int \left(\phi^{2\log n}\right)^2\mathrm{d}\lambda
=n\cdot \sum_{k=1}^{\left\lfloor2\log n\right\rfloor}k^2/2^k\\
&\leq n\cdot \sum_{k=1}^{\infty}k^2/2^k=6n.
\end{align*}
Applying this on \eqref{eq: lambda sum phi bernstein} gives 
\begin{align*}
 \lambda\left(\Upsilon_n\cap\Xi_n\right)
 &\leq  \exp\left(-\frac{n^{5/4}}{12n+\frac{2}{3}\cdot n^{5/8}\cdot 2\log n}\right)
 \leq  \exp\left(-\frac{n^{5/4}}{13 n}\right)
 \leq  \exp\left(-n^{1/8}\right),
\end{align*}
for $n$ sufficiently large.
Since $\sum_{n=1}^{\infty}\exp\left(-n^{1/8}\right)<\infty$,
applying the first Borel-Cantelli lemma 
yields
\begin{align}
\lambda\left(\Upsilon_n\cap \Xi_n~\text{ i.o.}\right)=0.\label{eq: Ups Xi io 0}
\end{align}

In the next steps we calculate the second summand of \eqref{eq: summand number>1}. We have, for $x\in\Omega'$, that
$\phi\circ \tau_{B}^{k-1}(x)>2\log n$ is equivalent to 
$\tau_{B}^{k-1}(x)\in\bigcup_{j=\lfloor 2\log n\rfloor-1}^{\infty}J_{j}$,
see the proof of Lemma \ref{lem: lambda tauB invar},
and this is equivalent to $\beta_k(x)\geq 2^{\lfloor 2\log n\rfloor-1}$.
As $n^{2\log 2}/4\leq 2^{\lfloor 2\log n\rfloor-1}$, we obtain
\begin{align*}
 \Xi_n^c
 &\subset\left\{\#\left\{k\leq n\colon \beta_k\geq n^{2\log 2}/4\right\}\geq 1\right\}
 =\left\{\#\left\{k\leq n\colon \beta_k\geq n\cdot n^{2\log 2-1}/4\right\}\geq 1\right\}\\
 &\subset \left\{\#\left\{k\leq n\colon \beta_k\geq n\cdot \mathrm{e}^{\left\lfloor \log n\right\rfloor\cdot\left(2\log 2-1\right)}/4\right\}\geq 1\right\}.
\end{align*}
If we set $\psi\left(n\right)= \mathrm{e}^{n\cdot\left(2\log 2-1\right)}/4$,
then $\psi\left(\left\lfloor\log n\right\rfloor\right)= \mathrm{e}^{\left\lfloor \log n\right\rfloor\cdot\left(2\log 2-1\right)}$
and $\psi\in\Psi$. Hence, we can apply Lemma \ref{lem: Xi> io 0} which yields that the second summand of \eqref{eq: summand number>1} equals zero.

Combining this with \eqref{eq: Ups Xi io 0} and \eqref{eq: summand number>1} yields the statement of the lemma.
\end{proof}

\begin{proof}[Proof of Lemma \ref{lem: Tn tn a.s.}]
We remember the definition of $\phi_n$ in \eqref{eq: def phi n}.
The strategy is to compare $\sum_{k=1}^n \eta^{2t_n}\circ \tau_B^k$ with $T_{\ell}^{t_n}$ with $\ell$ to be determined
and finally use Lemma \ref{lem: Sn* allg} to obtain the statement of the lemma.

Lemma \ref{lem: sum phi 2n} implies that we have eventually almost surely
\begin{align*}
2n-n^{3/4}\leq \phi_n\leq 2n+n^{3/4}.
\end{align*}
This yields that we have for every $\epsilon\in \left(0,1/2\right)$ eventually almost surely
\begin{align*}
\phi_{\left\lfloor n\left(1/2-\epsilon\right)\right\rfloor}
\leq n\cdot \left(1-2\epsilon\right)+\left(n\cdot \left(1/2-\epsilon\right)\right)^{3/4}
\leq n
\end{align*}
and 
\begin{align*}
 \phi_{\left\lfloor n\left(1/2+\epsilon\right)\right\rfloor}
 \geq 2\left\lfloor n\cdot \left(1/2+\epsilon\right)\right\rfloor-\left(\left\lfloor n\cdot \left(1/2+\epsilon\right)\right\rfloor\right)^{3/4}
 \geq n.
\end{align*}
Thus, we have for every $\epsilon\in(0,1/2)$ eventually almost surely
\begin{align*}
T_{\phi_{\left\lfloor n\left(1/2-\epsilon\right)\right\rfloor}}^{t_{n}}
\leq T_n^{t_n}
\leq T_{\phi_{\left\lfloor n\left(1/2+\epsilon\right)\right\rfloor}}^{t_n}.
\end{align*}
An easy calculation shows, for all $\epsilon\in(0,1/2)$, 
that $2t_{\left\lfloor n\left(1/2-\epsilon\right)\right\rfloor}\leq t_n\leq 2t_{\left\lfloor n\left(1/2+\epsilon\right)\right\rfloor}$ for $n$ sufficiently large. 
This implies 
\begin{align*}
T_{\phi_{\left\lfloor n\left(1/2-\epsilon\right)\right\rfloor}}^{2t_{\left\lfloor n\left(1/2-\epsilon\right)\right\rfloor}}
\leq T_n^{t_n}
\leq T_{\phi_{\left\lfloor n\left(1/2+\epsilon\right)\right\rfloor}}^{2t_{\left\lfloor n\left(1/2+\epsilon\right)\right\rfloor}}
\end{align*}
eventually almost surely.
We remember that 
\[
T^{r}_{\phi_n}
=\sum_{k=0}^{\phi(n)-1}\chi^{r}\circ \tau^k
=\sum_{k=0}^{n-1} \eta^{r}\circ \tau_B^k
=\mathsf{T}_n^{r}\eta, 
\]
which implies
\begin{align*}
\mathsf{T}_{\left\lfloor n\left(1/2-\epsilon\right)\right\rfloor}^{2t_{\left\lfloor n\left(1/2-\epsilon\right)\right\rfloor}}\eta
\leq T_n^{t_n}
\leq \mathsf{T}_{\left\lfloor n\left(1/2+\epsilon\right)\right\rfloor}^{2t_{\left\lfloor n\left(1/2+\epsilon\right)\right\rfloor}}\eta
\end{align*}
eventually almost surely.
Using Lemma \ref{lem: Sn* allg} implies, for all $\epsilon\in(0,1/2)$, eventually almost surely
\begin{align*}
\mathbb{E}\left(\mathsf{T}_{\left\lfloor n\left(1/2-\epsilon\right)\right\rfloor}^{2t_{\left\lfloor n\left(1/2-\epsilon\right)\right\rfloor}}\eta\right)\cdot\left(1-\epsilon\right)
\leq T_n^{t_n}
\leq \mathbb{E}\left(\mathsf{T}_{\left\lfloor n\left(1/2+\epsilon\right)\right\rfloor}^{2t_{\left\lfloor n\left(1/2+\epsilon\right)\right\rfloor}}\eta\right)\cdot\left(1+\epsilon\right).
\end{align*}
Furthermore, since by Lemma \ref{lem: lambda tauB invar} $\lambda$ is $\tau_B$-invariant,
Lemma \ref{lem: rel w eta ell} implies
that we have for every $\epsilon\in(0,1/2)$ and $m\in\mathbb{N}$ sufficiently large eventually almost surely
\begin{align*}
 T_n^{t_n}
 &\geq \left\lfloor n\left(1/2-\epsilon\right)\right\rfloor\cdot \mathbb{E}\left(\eta^{2t_{\left\lfloor n\left(1/2-\epsilon\right)\right\rfloor}}\right)\cdot\left(1-\epsilon\right)
 \geq  n\cdot \mathbb{E}\left(w_m^{2t_{\left\lfloor n\left(1/2-\epsilon\right)\right\rfloor}}\right)\cdot\left(1/2-2\epsilon\right).
\end{align*}
Choosing $m$ sufficiently large and applying Lemma \ref{lem: EW wm} for every $\epsilon\in(0,1/2)$ eventually almost surely
\begin{align}
 T_n^{t_n}
 &\geq n\cdot\log (2t_{\left\lfloor n\left(1/2-\epsilon\right)\right\rfloor})\cdot \left(1-3\epsilon\right)
 \geq n\cdot\log \left(n\left(1-2\epsilon\right)\right)\cdot \left(1-3\epsilon\right)
 \geq n\cdot\log n\cdot \left(1-4\epsilon\right).\label{eq: Tntn>}
\end{align}
On the other hand, by an analogous combination 
of Lemma \ref{lem: rel w eta ell}
and Lemma \ref{lem: EW vm},
we have for every $\epsilon\in(0,1/2)$ eventually almost surely
\begin{align}
 T_n^{t_n}
 &\leq \left\lfloor n\left(1/2+\epsilon\right)\right\rfloor\cdot \mathbb{E}\left(\eta^{2t_{\left\lfloor n\left(1/2+\epsilon\right)\right\rfloor}}\right)\cdot\left(1+\epsilon\right)
 \leq n\cdot \mathbb{E}\left(v_m^{2t_{\left\lfloor n\left(1/2-\epsilon\right)\right\rfloor}}\right)\cdot\left(1/2+2\epsilon\right)\notag\\
 &\leq n\cdot\log (2t_{\left\lfloor n\left(1/2+\epsilon\right)\right\rfloor})\cdot \left(2+\epsilon\right)\cdot\left(1/2+2\epsilon\right)
 \leq n\cdot\log (2t_{\left\lfloor n\left(1/2+\epsilon\right)\right\rfloor})\cdot \left(1+6\epsilon\right)\notag\\
 &\leq n\cdot\log n^{1+\epsilon}\cdot \left(1+7\epsilon\right)
 \leq n\cdot\log n\cdot\left(1+12\epsilon\right).\label{eq: Tntn<}
\end{align}
Combining \eqref{eq: Tntn>} and \eqref{eq: Tntn<} gives the statement of the lemma.
\end{proof}

\section{Proof of the weak convergence Theorem \ref{thm: weak conv}}\label{sec: proof weak conv}
\subsection{Mixing properties of the digits \texorpdfstring{$(a_n)$}{an}}\label{subsec: mixin prop}
In contrast to the proof of Theorem \ref{thm: if part} 
we don't use the independence properties of the induced transformation but show that 
that the digits $(a_n)$ are $\bm{\alpha}$-mixing which enables us to prove Theorem \ref{thm: weak conv}.
So we will first give the definition of $\bm{\alpha}$-mixing random variables. 
\begin{Def}
Let $\left(\Omega',\mathcal{A},\mathbb{P}\right)$ be a probability measure space and $\mathcal{B},\mathcal{C}\subset\mathcal{A}$ two $\sigma$-fields, then the following measure of dependence is defined.
\begin{align*}
\bm{\alpha}\left(\mathcal{B},\mathcal{C}\right)&=\sup_{B,C}\left|\mathbb{P}\left(B\cap C\right)-\mathbb{P}\left(B\right)\cdot \mathbb{P}\left(C\right)\right|&&B\in\mathcal{B}, C\in\mathcal{C}.
\end{align*}

Furthermore, let $\left(X_{n}\right)_{n\in\mathbb{N}}$ be a (not necessarily stationary) sequence of random variables. For $-\infty\leq J\leq L\leq\infty$ we can define a $\sigma$-field by
\begin{align*}
\mathcal{A}_{J}^{L}=\sigma\left(X_{k},J\leq k\leq L, k\in\mathbb{Z}\right).
\end{align*}
With that the dependence coefficient is defined by
\begin{align}
\bm{\alpha}\left(n\right)&=\sup_{k\in\mathbb{N}}\bm{\alpha}\left(\mathcal{A}_{1}^{k},\mathcal{A}_{k+n}^{\infty}\right),\notag
\end{align}
The sequence $\left(X_{n}\right)$ is said to be $\bm{\alpha}$-mixing if $\lim_{n\to\infty}\bm{\alpha}\left(n\right)= 0$.
\end{Def}
For further properties of mixing random variables see \cite{bradley_basic_2005}.

\begin{lemma}\label{lem: Xn alpha mixing}
The sequence $\left(a_n\right)$ is $\bm{\alpha}$-mixing.
\end{lemma}

\begin{proof}
The proof is based on a decay of correlation argument for the transfer operator going back to classical results,
see for example \cite{baladi_positive_2000}.
Since the proof in this case is reasonably short we redo it as a special case.
In order to proceed we first need the notion of bounded variation.
\begin{Def}
For $\varphi:\left[0,1\right)\to\mathbb{R}_{\geq 0}$ the variation $\mathsf{V}\left(\varphi\right)$ is given by
\begin{align*}
\mathsf{V}\left(\varphi\right)=\sup\left\{\sum_{i=1}^n\left|\varphi\left(x_i\right)-\varphi\left(x_{i-1}\right)\right|\colon n\geq 1, x_i\in\left[0,1\right), x_0<x_1<\ldots<x_n\right\} .
\end{align*}
By $BV$ we denote the Banach space of functions of bounded variation, i.e.\ of functions $\varphi$ fulfilling $\mathsf{V}\left(\varphi\right)<\infty$. It is equipped with the norm
$\left\Vert\varphi\right\Vert_{BV}=\mathsf{V}\left(\varphi\right)+\left\Vert\varphi\right\Vert_{\infty}$.
\end{Def}
For further properties of functions of bounded variation see for example \cite[Chapter 2]{boyarsky_laws_1997}.

We define the transfer operator $\widehat{\tau}$ as the uniquely up to a.s.\ equivalence defined operator such that for all $\varphi\in L^{\infty}$ and all $\zeta\in L^1$ it holds that
\begin{align*}
 \int \left(\varphi\circ \tau\right)\cdot \zeta\mathrm{d}\lambda=\int\varphi\cdot \widehat{\tau}\zeta\mathrm{d}\lambda.
\end{align*}
For every $\varphi\in L^{\infty}$ we have that
\begin{align*}
\int_0^1\varphi\left(\tau x\right)\cdot\zeta\left(x\right)\mathrm{d}\lambda\left(x\right)
&=\int_0^{1/2}\varphi\left(2x\right)\cdot\zeta\left(x\right)\mathrm{d}\lambda\left(x\right)+\int_{1/2}^1\varphi\left(2x-1\right)\cdot\zeta\left(x\right)\mathrm{d}\lambda\left(x\right).
\end{align*}
Setting $s=2x$ in the first summand and $t=2x-1$ in the second summand yields
\begin{align*}
\int_0^1\varphi\left(\tau x\right)\cdot\zeta\left(x\right)\mathrm{d}\lambda\left(x\right)
&=\int_0^{1}\varphi\left(s\right)\cdot\zeta\left(\frac{s}{2}\right)\mathrm{d}\lambda\left(\frac{s}{2}\right)+\int_0^{1}\varphi\left(t\right)\cdot\zeta\left(\frac{t+1}{2}\right)\mathrm{d}\lambda\left(\frac{t+1}{2}\right)\\
&=\int_0^{1}\varphi\left(x\right)\cdot\frac{1}{2}\cdot \left(\zeta\left(\frac{x}{2}\right)+\zeta\left(\frac{x+1}{2}\right)\right)\mathrm{d}\lambda\left(x\right).
\end{align*}
Thus,
\begin{align*}
\left(\widehat{\tau}\zeta\right)\left(x\right)=\frac{1}{2}\cdot\left(\zeta\left(\frac{x}{2}\right)+\zeta\left(\frac{x+1}{2}\right)\right).
\end{align*}
To obtain some decay of correlation we first estimate now
\begin{align*}
\mathsf{V}\left(\widehat{\tau}\zeta\right)&=\sup\sum_{i=1}^n\left|\left(\widehat{\tau}\zeta\right)\left(x_i\right)-\left(\widehat{\tau}\zeta\right)\left(x_{i-1}\right)\right|\\
&=\sup\sum_{i=1}^n\frac{1}{2}\cdot\left|\zeta\left(\frac{x_i}{2}\right)+\zeta\left(\frac{x_i+1}{2}\right)-\zeta\left(\frac{x_{i-1}}{2}\right)-\zeta\left(\frac{x_{i-1}+1}{2}\right)\right|,
\end{align*}
where the supremum is taken over $n\in\N_0$ and $x_i\in[0,1)$ such that $x_0<\ldots<x_n$.
By renaming $x_i/2=y_{i}$ and $\left(x_i+1\right)/2=y_{n+i}$ we obtain
\begin{align*}
\MoveEqLeft\sum_{i=1}^n\left|\zeta\left(\frac{x_i}{2}\right)+\zeta\left(\frac{x_i+1}{2}\right)-\zeta\left(\frac{x_{i-1}}{2}\right)-\zeta\left(\frac{x_{i-1}+1}{2}\right)\right|\\
&=\sum_{i=1}^n\left|\zeta\left(y_i\right)+\zeta\left(y_{n+i}\right)-\zeta\left(y_{i-1}\right)-\zeta\left(y_{n+i-1}\right)\right|\\
&\leq \sum_{i=1}^{2n}\left|\zeta\left(y_i\right)-\zeta\left(y_{i-1}\right)\right|.
\end{align*}
Thus,
\begin{align}
\mathsf{V}\left(\widehat{\tau}\zeta\right)&=\sup\left\{\sum_{i=1}^n\left|\left(\widehat{\tau}\zeta\right)\left(x_i\right)-\left(\widehat{\tau}\zeta\right)\left(x_{i-1}\right)\right|\colon n\in\mathbb{N}, x_i\in\left[0,1\right), x_1<\ldots<x_n\right\}\notag\\
&\leq \frac{1}{2}\sup\left\{\sum_{i=1}^m\left|\zeta\left(z_i\right)-\zeta\left(z_{i-1}\right)\right|\colon m\in\mathbb{N}, z_i\in\left[0,1\right), z_1<\ldots<z_m\right\}
=\frac{1}{2}\mathsf{V}\left(\zeta\right).\label{eq: hat tau psi}
\end{align}

Furthermore, we can decompose the space of $BV$-functions in $BV=P\oplus H$, where $P$ is the projective space $\mathbb{C}\mathbbm{1}$ and $H=\left\{\zeta\in BV\colon \int\zeta\mathrm{d}\lambda=0\right\}$. Each $\zeta\in BV$ can be written as $\zeta=\int\zeta\mathrm{d}\lambda+\zeta_H$, where $\int\zeta_H\mathrm{d}\lambda=0$.
The decomposition of $BV$ is invariant under $\widehat{\tau}$ since $\widehat{\tau}\mathbbm{1}=\mathbbm{1}$ and for $\int\zeta_H\mathrm{d}\lambda=0$ it holds that $\int\widehat{\tau}\zeta_H\mathrm{d}\lambda=\int\left(\mathbbm{1}\circ\tau\right)\zeta_H\mathrm{d}\lambda=\int\zeta_H\mathrm{d}\lambda=0$.

To obtain a decay of correlation result we notice that an iterated application of the definition of the transfer operator yields
\begin{align*}
 \int\left(\varphi\circ \tau^k\right)\cdot\zeta\mathrm{d}\lambda=\int\varphi\cdot\left(\widehat{\tau}^k\zeta\right)\mathrm{d}\lambda.
\end{align*}
The decay of correlation is then estimated by
\begin{align}
\mathrm{Cor}_n\left(\varphi,\zeta\right)
&=\left|\int\left(\varphi\circ\tau^n\right)\cdot\zeta\mathrm{d}\lambda-\int\varphi\mathrm{d}\lambda\cdot\int\zeta\mathrm{d}\lambda\right|
=\left|\int\varphi\cdot\left(\widehat{\tau}^n\zeta\right)\mathrm{d}\lambda-\int\varphi\mathrm{d}\lambda\cdot\int\zeta\mathrm{d}\lambda\right|\notag\\
&=\left|\int\varphi\cdot\widehat{\tau}^n\left(\int\zeta\mathrm{d}\lambda+\zeta_H\right)\mathrm{d}\lambda-\int\varphi\mathrm{d}\lambda\cdot\int\zeta\mathrm{d}\lambda\right|
=\left|\int\varphi\cdot\widehat{\tau}^n\zeta_H\mathrm{d}\lambda\right|\notag\\
&\leq \left\Vert\varphi\right\Vert_1\cdot\left\Vert\zeta_H\right\Vert_{\infty},\label{eq: dec cor}
\end{align}
where we used the fact that $\widehat{\tau}^n\mathbbm{1}=\mathbbm{1}$. Since $\widehat{\tau}^n\zeta_H\in H$, it follows that its range has a diameter less or equal to $\mathsf{V}\left(\zeta\right)$ and contains zero in the convex hull. Thus, $\left\Vert\widehat{\tau}^n\zeta_H\right\Vert_{\infty}\leq\mathsf{V}\left(\widehat{\tau}^n\zeta_H\right)$ and by \eqref{eq: hat tau psi} it follows that
\begin{align*}
\left\Vert\widehat{\tau}^n\zeta_H\right\Vert_{\infty}\leq \mathsf{V}\left(\widehat{\tau}^n\zeta_H\right)\leq 2^{-n}\cdot\mathsf{V}\left(\zeta_H\right)=2^{-n}\cdot\mathsf{V}\left(\zeta\right).
\end{align*}
Combining this with \eqref{eq: dec cor} yields for all $\varphi\in L^1$ and all $\zeta\in BV$ that
\begin{align}
\mathrm{Cor}_n\left(\varphi,\zeta\right)\leq 2^{-n}\cdot\left\Vert\varphi\right\Vert_1\cdot\mathsf{V}\left(\zeta\right). \label{eq: cor}
\end{align}
We further note that each $a_i$ can only take values in the natural numbers.
To prove $\bm{\alpha}$-mixing we first notice that for all $i,k,n\in\mathbb{N}$ and $A\in\sigma\left(\N\right)$ we have that
\begin{align*}
\lambda\left(\left\{a_k=i\right\}\cap\left\{a_{n+k}\in A\right\}\right)
&=\int\mathbbm{1}_{\left\{a_k=i\right\}}\cdot\left(\mathbbm{1}_{A}\circ\tau^{n+k-1}\right)\mathrm{d}\lambda\\
&=\int\left(\mathbbm{1}_{\left\{\chi=i\right\}}\circ\tau^{k-1}\right)\cdot\left(\mathbbm{1}_{A}\circ\tau^{n+k-1}\right)\mathrm{d}\lambda\\
&=\int\mathbbm{1}_{\left\{a_1=i\right\}}\cdot\left(\mathbbm{1}_{A}\circ\tau^{n-1}\right)\mathrm{d}\lambda.
\end{align*}
Obviously, $\left\Vert\mathbbm{1}_{A}\right\Vert_1\leq 1$ and thus $\mathbbm{1}_{A}\in L^1$ and further $\mathsf{V}\left(\mathbbm{1}_{\left\{a_k=i\right\}}\right)\leq 2$ and thus $\mathbbm{1}_{\left\{a_1=i\right\}}\in BV$. Applying \eqref{eq: cor} and the fact that $\left\Vert\mathbbm{1}_A\right\Vert_1=\lambda\left(A\right)$ yields
\begin{align}
\left|\lambda\left(\left\{a_k=i\right\}\cap\left\{a_{n+k}\in A\right\}\right)-\lambda\left(a_k=i\right)\lambda\left(a_{n+k}\in A\right)\right|\leq 2^{-n}.\label{eq: Xnk A Xk i}
\end{align}
Since every $B\in\sigma\left(\N\right)$ is just any subset $I\subset \N$, we have for each $k,n\in\N$
\begin{align}
\MoveEqLeft\left|\lambda\left(\left\{a_k\in B\right\}\cap\left\{a_{n+k}\in A\right\}\right)-\lambda\left(a_k\in B\right)\lambda\left(a_{n+k}\in A\right)\right|\notag\\
&=\left|\sum_{i\in I}\big(\lambda\left(\left\{a_k=i\right\}\cap\left\{a_{n+k}\in A\right\}\right)-\lambda\left(a_k=i\right)\lambda\left(a_{n+k}\in A\right)\big)\right|\notag\\
&\leq \left|\sum_{i\in I\colon i\leq \lfloor 2^{n/2}\rfloor}\big(\lambda\left(\left\{a_k=i\right\}\cap\left\{a_{n+k}\in A\right\}\right)-\lambda\left(a_k=i\right)\lambda\left(a_{n+k}\in A\right)\!\big)\right|\label{eq: Xnk in A0}\\
&\qquad+\lambda\left(a_k\geq \left\lfloor2^{n/2}\right\rfloor\right).\label{eq: Xnk in A}
\end{align}
By \eqref{eq: Xnk A Xk i} we can estimate the sum in \eqref{eq: Xnk in A0}, for each $k,n\in\N$, by
\begin{align}
\sum_{i\in I\colon i\leq \lfloor 2^{n/2}\rfloor}\left(\lambda\left(\left\{a_k=i\right\}\cap\left\{a_{n+k}\in A\right\}\right)-\lambda\left(a_k=i\right)\lambda\left(a_{n+k}\in A\right)\right)
\leq \left\lfloor2^{n/2}\right\rfloor\cdot 2^{-n}
\leq 2^{-n/2}.\label{eq: Xnk in A 1}
\end{align}
Using the distribution function in \eqref{eq: F dist func} the summand in \eqref{eq: Xnk in A} can, for each $k\in\N$, be estimated by
\begin{align}
\lambda\left(a_k\geq \left\lfloor2^{n/2}\right\rfloor\right)=\frac{1}{\left\lfloor2^{n/2}\right\rfloor}\leq 2^{-n/2+1}.\label{eq: Xnk in A 2}
\end{align}
Combining \eqref{eq: Xnk in A 1} and \eqref{eq: Xnk in A 2} yields, for all $k,n\in\N$,
\begin{align*}
\left|\lambda\left(\left\{a_k\in B\right\}\cap\left\{a_{n+k}\in A\right\}\right)-\lambda\left(a_k\in B\right)\lambda\left(a_{n+k}\in A\right)\right|\leq 2^{-n/2+2}.
\end{align*}
\end{proof}

\subsection{Proof of Theorem \ref{thm: weak conv}}
To prove Theorem \ref{thm: weak conv} we need additionally to the mixing properties of the digits some auxiliary definitions and lemmas.
\begin{Def}[Property $\mathbf{B}$]
Let $\left(Y_n\right)$ be a sequence of strictly stationary random variables and $Z_n=\sum_{k=1}^nY_k$.
We say that Property $\mathbf{B}$ is fulfilled for a sequence of constants $\left(B_n\right)_{n\in\mathbb{N}}$ if
\begin{align*}
 \lim_{n\to\infty}\max \left|\mathbb{E}\left(\exp\left(it\cdot\frac{Z_{k+\ell}}{B_n}\right)\right)-\mathbb{E}\left(\exp\left(it\cdot\frac{Z_{k}}{B_n}\right)\right)\cdot\mathbb{E}\left(\exp\left(it\cdot\frac{Z_{\ell}}{B_n}\right)\right)\right|=0
\end{align*}
for all $t\in\mathbb{R}$, where the maximum is taken over $k,\ell\in\N$ fulfilling $1\leq k+\ell\leq n$.
\end{Def}

The following lemma will give a criterion for convergence in probability for a sum of truncated normed random variables.
\begin{lemma}[{{\cite[Theorem 2]{szewczak_relative_2001}}}]\label{lem: szewczak}
Let $\left(Y_n\right)_{n\in\mathbb{N}}$ be a sequence of non-negative, identically distributed random variables 
and for $r>0$ set $U_n^r\coloneqq\sum_{k=1}^nY_k\cdot\mathbbm{1}_{\left\{Y_k\leq r\right\}}$.
Furthermore, assume that the following hold:
\begin{enumerate}[(a)]
 \item\label{en: a} There exists a positive valued sequence $\left(f_n\right)$ tending to infinity such that Property $\mathbf{B}$ is fulfilled for $\mathbb{E}\left(U_n^{f_n}\right)$,
 \item\label{en: b} we have that
 \begin{align*}
\lim_{n\to\infty}\frac{\sum_{k=1}^nY_k\cdot\mathbbm{1}_{\left\{Y_k>f_n\right\}}}{\mathbb{E}\left(U_n^{f_n}\right)}=0
\end{align*}
in probability, and
 \item\label{en: c} $\left(U_n^{f_n}/\mathbb{E}\left(U_n^{f_n}\right)\right)$ is uniformly integrable.
\end{enumerate}

Then
\begin{align*}
\lim_{n\to\infty}\frac{U_n^{f_n}}{\mathbb{E}\left(U_n^{f_n}\right)}=1
\end{align*}
in probability.
\end{lemma}

The next two lemmas will enable us to apply Lemma \ref{lem: szewczak}.
\begin{lemma}[{{\cite[Lemma 5.2]{jakubowski_minimal_1993}}}]\label{lem: jakubowski}
If $\left(Y_n\right)_{n\in\mathbb{N}}$ is a strictly stationary, $\bm{\alpha}$-mixing process, then Property $\mathbf{B}$ holds for all sequences $\left(B_n\right)$ tending to infinity.
\end{lemma}

\begin{lemma}[de la Vall{\'e}e-Poissin's criterion (see {\cite[II,22]{dellacherie_probabilities_1978}})]\label{lem: Vallee-Poussin}
The family of random variables $\left(Y_n\right)_{n\in \N}$ is uniformly integrable if and only if there exists a non-decreasing, convex, continuous function $h:\mathbb{R}_{> 0}\to \mathbb{R}_{> 0}$ fulfilling $\lim_{x\to\infty}h\left(x\right)/x=\infty$ and
\begin{align}
 \sup_{n\in \N} \mathbb{E}\left(h\left(Y_n\right)\right)<\infty.\label{eq: vallee poussin}
\end{align}
\end{lemma}

With this information at hand we are able to prove Theorem \ref{thm: weak conv}.
\begin{proof}[Proof of Theorem \ref{thm: weak conv}]
The proof of this theorem will be done in two steps. Setting $r_n=n\log n$,
the first step is to prove that
$\lim_{n\to\infty}\lambda\left(S_n>T_n^{r_n}\right)=0$.
In the second step we will prove with the help of Lemma \ref{lem: szewczak} that 
$\lim_{n\to\infty}T_n^{r_n}/(n\log n)=1$ in probability.

First we note that by the $\tau$-invariance of $\lambda$
\begin{align*}
 \lambda\left(S_n>T_n^{r_n}\right)
 &=\lambda\left(\bigcup_{k=1}^n\left\{a_k>r_n\right\}\right)
 \leq \sum_{k=1}^n\lambda\left(a_k>r_n\right)
 =n\cdot \lambda\left(a_1>r_n\right).
\end{align*}
Using the distribution of $a_1$ given in \eqref{eq: F dist func} gives 
\begin{align}
\lim_{n\to\infty}\lambda\left(S_n>T_n^{r_n}\right)
\leq \lim_{n\to\infty}\frac{n}{\left\lfloor r_n\right\rfloor+1}
=\lim_{n\to\infty}\frac{1}{\log n}=0.\label{eq: lim lambda Sn>Tn}
\end{align}

In order to prove the second part of the theorem 
we aim to apply Lemma \ref{lem: szewczak} on $Y_k=a_k$ and $\left(f_n\right)=\left(r_n\right)$.
First we notice that $(a_n)$ are strictly stationary and by Lemma \ref{lem: jakubowski} and Lemma \ref{lem: Xn alpha mixing} we have that Condition $\mathbf{B}$ is fulfilled since $\left(r_n\right)$ tends to infinity.
This gives us \ref{en: a} of Lemma \ref{lem: szewczak}.

Furthermore, by \eqref{eq: E Tn tn} we have for $n$ sufficiently large that
\begin{align*}
\lambda\left(\left|\frac{\sum_{k=1}^na_k\cdot\mathbbm{1}_{\left\{a_k>r_n\right\}}}{\mathbb{E}\left(T_n^{r_n}\right)}\right|>\epsilon\right)
 &\leq \lambda\left(\sum_{k=1}^na_k\cdot\mathbbm{1}_{\left\{a_k>r_n\right\}}>\epsilon/2\cdot n\cdot\log r_n\right)\\
 &\leq \lambda\left(\sum_{k=1}^n\mathbbm{1}_{\left\{a_k>r_n\right\}}\geq 1\right)
 =\lambda\left(S_n>T_n^{r_n}\right),
\end{align*}
which by \eqref{eq: lim lambda Sn>Tn} tends to zero and hence \ref{en: b} holds.

Finally, to prove the uniform integrability of $\left(T_n^{r_n}/\mathbb{E}\left(T_n^{r_n}\right)\right)$ we use Lemma \ref{lem: Vallee-Poussin} 
for $Y_n=T_n^{r_n}/\mathbb{E}\left(T_n^{r_n}\right)$ and choose $h$ as $h\left(x\right)=x^2$. 
We have that
\begin{align}
\mathbb{E}\left(\left(T_n^{r_n}\right)^2\right)
&=\sum_{i,j=1}^n\mathbb{E}\left(a_i^{r_n}\cdot a_j^{r_n}\right)
=\sum_{k=1}^n\mathbb{E}\left(\left(a_k^{r_n}\right)^2\right)+2\sum_{1\leq i<j\leq n}\mathbb{E}\left(a_i^{r_n}\cdot a_j^{r_n}\right).\label{eq: E Tntn ^2}
\end{align}
For the first summands in \eqref{eq: E Tntn ^2} we have by stationarity and the distribution function given in \eqref{eq: F dist func}, for all $k\in\N$, that
\begin{align*}
\mathbb{E}\left(\left(a_k^{r_n}\right)^2\right)
&=\sum_{i=1}^{\left\lfloor r_n\right\rfloor}\lambda\left(a_k=i\right)\cdot i^2
=\sum_{i=1}^{\left\lfloor r_n\right\rfloor}\left(\frac{1}{i}-\frac{1}{i+1}\right)\cdot i^2
=\sum_{i=1}^{\left\lfloor r_n\right\rfloor}\frac{i}{i+1}
<r_n
\end{align*}
and the choice of $\left(r_n\right)$ yields
\begin{align}
\sum_{k=1}^n \mathbb{E}\left(\left(a_k^{r_n}\right)^2\right)<n^2\left(\log n\right).\label{eq: E Tntn ^2 1}
\end{align}
To estimate the second sum in \eqref{eq: E Tntn ^2} we 
notice that 
\begin{align}
 \mathbb{E}\left(a_i^{r_n}\cdot a_j^{r_n}\right)
&=\int\left(\chi^{r_n}\circ \tau^{i-1}\right)\cdot \left(\chi^{r_n}\circ \tau^{j-1}\right)\mathrm{d}\lambda
=\int\chi^{r_n}\cdot \left(\chi^{r_n}\circ \tau^{j-i}\right)\mathrm{d}\lambda\notag\\
&\leq \left(\int \chi^{r_n}\mathrm{d}\lambda\right)^2+\mathrm{Cor}_{j-i}\left(\chi^{r_n},\chi^{r_n}\right).\label{eq: E ai aj}
\end{align}
For the first summand we have by \eqref{eq: E X tn}
\begin{align}
 \left(\int \chi^{r_n}\mathrm{d}\lambda\right)^2
 &\leq 2\left(\log r_n\right)^2
 =\left(\log \left(n\cdot \log n\right)\right)^2
 \leq 4\left(\log n\right)^2,\label{eq: E ai aj 1}
\end{align}
for $n$ sufficiently large
and for the second summand of \eqref{eq: E ai aj} we have by \eqref{eq: cor} and \eqref{eq: E X tn}
\begin{align}
 \mathrm{Cor}_{j-i}\left(\chi^{r_n},\chi^{r_n}\right)
 &\leq 2^{-j+i}\cdot \left\Vert\chi^{r_n}\right\Vert_1\cdot \mathsf{V}\left(\chi^{r_n}\right)
 \leq 2^{-j+i+1}\cdot \log r_n\cdot r_n\notag\\
 &\leq 2^{-j+i+1}\cdot\left(\log n+\log\log n\right)\cdot n\cdot\log n
 \leq  2^{-j+i+2}\cdot n\cdot\left(\log n\right)^2,\label{eq: E ai aj 2}
\end{align}
for $n$ sufficiently large.

Combining \eqref{eq: E ai aj} with \eqref{eq: E ai aj 1} and \eqref{eq: E ai aj 2} yields
\begin{align}
\sum_{1\leq i<j\leq n}\mathbb{E}\left(a_i^{r_n}\cdot a_j^{r_n}\right)
&\leq  \sum_{i=1}^n\left( 4\left(\log n\right)^2+\sum_{j>i} 2^{-j+i+2}\cdot  n\cdot\left(\log n\right)^2\right)\notag\\
&=4 n\cdot \left(\log n\right)^2+4 n^2\cdot\left(\log n\right)^2
\leq 8n^2\cdot \left(\log n\right)^2,\label{eq: E Tntn ^2 2}
\end{align}
for $n$ sufficiently large.
Combining \eqref{eq: E Tntn ^2} with
\eqref{eq: E Tntn ^2 1} and \eqref{eq: E Tntn ^2 2} yields
$\mathbb{E}\big(\left(T_n^{r_n}\right)^2\big)\leq 9n^2\left(\log n\right)^2$,
for $n$ sufficiently large.

On the other hand, applying \eqref{eq: E Tn tn} yields
\begin{align}
\mathbb{E}\left(T_n^{r_n}\right)\sim n\cdot\log r_n = n\cdot\left(\log n+\log\log n\right)\sim n\cdot\log n\label{eq: E Tn tn 1}
\end{align}
and thus
\begin{align*}
\lim_{n\to\infty}\mathbb{E}\left(\left(\frac{T_n^{r_n}}{\mathbb{E}\left(T_n^{r_n}\right)}\right)^2\right)\leq9<\infty.
\end{align*}
Hence, \eqref{eq: vallee poussin} follows and by Lemma \ref{lem: Vallee-Poussin} \ref{en: c} holds. 
Hence, Lemma \ref{lem: szewczak} is applicable giving the weak convergence $\lim_{n\to\infty}T_n^{r_n}/\mathbb{E}\left(T_n^{r_n}\right)=1$ in probability.

Lastly, \eqref{eq: E Tn tn 1} gives the denominator in \eqref{eq: Sn log n distr}.
\end{proof}

\section{Proof of Theorem \ref{thm: only if part}}\label{sec: proof only if}
\begin{proof}[Proof of Theorem \ref{thm: only if part}]
Assume that $(b_n)$ is as in \eqref{eq: cond bn 1} with $\psi\in\overline{\Psi}$.
The strategy of the proof is to show
for arbitrary $u\in\N$ that
\begin{align}
\lambda\left(\#\left\{i\leq n\colon a_i\geq n\cdot \log n\right\}\geq b_n+u~\text{ i.o.}\right)=1\label{eq: ai>nlogn io}
\end{align}
which implies
$\lambda\left(S_{n}^{b_{n}}\geq n\cdot \log n\cdot u~\text{ i.o.}\right)=1$.
Noting that $u$ can be chosen arbitrarily large gives the statement of the theorem. 

To show \eqref{eq: ai>nlogn io}
we set $(I_k)$ as in \eqref{eq: Kk}.
We note that for $n\in I_j$, we firstly have
$n\geq 2^j$ and secondly $n\cdot \log n< 2^{j+1}\cdot \log 2^{j+1}<2^{j+2}\cdot j$, for $j$ sufficiently large. 
For $n\in I_j$ and $j$ sufficiently large the definition of $(b_n)$ in \eqref{eq: cond bn 1} and this calculation yields
\begin{align}
 \MoveEqLeft\left\{\#\left\{i\leq n\colon a_i\geq n\cdot \log n\right\}\geq b_n+u\right\}\notag\\
 &= \left\{\#\left\{i\leq n\colon a_i\geq n\cdot \log n\right\}\geq \left\lfloor\frac{\log\psi\left(\left\lfloor \log n\right\rfloor\right)-\log\log n}{\log 2}\right\rfloor+u\right\}\notag\\
 &\supset \left\{\#\left\{i\leq 2^{j}\colon a_i\geq 2^{j+2}\cdot j\right\}\geq \left\lfloor\frac{\log\left(2^{u+1}\psi\left(\left\lfloor \log n\right\rfloor\right)\right)-\log\log n}{\log 2}\right\rfloor\right\}.\label{eq: number ileq n supset}
\end{align}

In order to proceed we need the following lemma which is an analog of Lemma \ref{lem: log gamma log tilde gamma}.
\begin{lemma}\label{lem: log gamma log tilde gamma overline}
Let $\psi\in\overline{\Psi}$. Then there exists $\omega\in\overline{\Psi}$ such that
\begin{align}
\omega\left(\left\lfloor \log_2 n\right\rfloor \right)\geq \psi\left(\left\lfloor \log n\right\rfloor\right).\label{phi psi a b1}
\end{align}
\end{lemma}
\begin{proof}
 The proof can be done analogously to the proof of Lemma \ref{lem: log gamma log tilde gamma}.
 We define $\omega:\mathbb{N}\rightarrow\mathbb{R}_{>0}$ as
\begin{align}
\omega\left(n\right)=\max\left\{ \psi\left(\left\lfloor n\cdot \log 2\right\rfloor +j\right)\colon j\in\left\{ 0,1\right\} \right\}.\label{eq: omega n}
\end{align}
Recall that $\psi\in\overline{\Psi}$. Then for the functions
$\overline{\psi}:\mathbb{N}\to\mathbb{R}_{>0}$ and $\widetilde{\psi}:\mathbb{N}\to\mathbb{R}_{>0}$ given by $\overline{\psi}\left(n\right)=\psi\left(\left\lfloor \kappa\cdot n\right\rfloor \right)$
with $\kappa>0$ and $\widetilde{\psi}\left(n\right)=\max\left\{ \psi\left(n\right),\psi\left(n+1\right)\right\} $
it holds that $\widetilde{\psi},\overline{\psi}\in\overline{\Psi}$.
Hence, $\omega\in\overline{\Psi}$.
Applying $\left\lfloor \log_2 n\right\rfloor $ on $\omega$ in \eqref{eq: omega n} yields
\begin{align*}
\omega\left(\left\lfloor \log_2 n\right\rfloor \right)=\max\left\{ \psi\left(\left\lfloor \left\lfloor \frac{\log n}{\log 2}\right\rfloor\cdot\log 2 \right\rfloor +j\right)\colon j\in\left\{ 0, 1 \right\} \right\} .
\end{align*}
Using \eqref{eq: omega estim 1} and \eqref{eq: omega estim 2} gives 
\begin{align*}
 \max\left\{ \psi\left(\left\lfloor \left\lfloor \frac{\log n}{\log 2}\right\rfloor\cdot\log 2 \right\rfloor +j\right)\colon j\in\left\{ 0,1 \right\} \right\}
 \geq \psi\left(\left\lfloor\log n\right\rfloor\right)
\end{align*}
and \eqref{phi psi a b1} follows.
\end{proof}

Noting that 
\begin{align*}
 \log\log n
 \geq \log\left(\left\lfloor\log_2 n\right\rfloor\cdot \log 2\right)
 = \log\left\lfloor\log_2 n\right\rfloor+\log\log 2
 \geq \log\left\lfloor\log_2 n\right\rfloor-\log 2
\end{align*}
and applying Lemma \ref{lem: log gamma log tilde gamma overline} yields
\begin{align}
 \left\lfloor\frac{\log\left(2^{u+1}\psi\left(\left\lfloor \log n\right\rfloor\right)\right)-\log\log n}{\log 2}\right\rfloor
 &\leq \left\lfloor \log_2\left(2^{u+2}\omega\left(\left\lfloor \log_2 n\right\rfloor\right)\right)-\log_2\left\lfloor\log_2 n\right\rfloor\right\rfloor.\label{eq: number ileq n supset 0}
\end{align}

We note that $\{1,\ldots, 2^j\}\supset I_{j-1}$ and for $n\in I_j$ we have $j=\lfloor\log_2 n\rfloor$ which implies
Inserting then \eqref{eq: number ileq n supset 0} in \eqref{eq: number ileq n supset} yields for $j$ sufficiently large
\begin{align}
\MoveEqLeft\left\{\#\left\{i\leq n\colon a_i\geq n\cdot \log n\right\}\geq b_n+u\right\}\notag\\
 &\supset \left\{\#\left\{i\in I_{j-1}\colon a_i\geq 2^{j+2}\cdot j\right\}\geq \left\lfloor\log_2\left(2^{u+2}\omega\left(j\right)\right)-\log_2 j\right\rfloor\right\}.\label{eq: number ileq n supset1}
\end{align}
For the following let 
\begin{align*}
 \overline{I}_{j}=\begin{cases}
                 \left[2^{j},2^{j+1}-\left\lfloor\log_2\left(2^{u+2}\omega\left(j+1\right)\right)\right\rfloor-1\right]&\text{ if }2^{j-1}\geq\left\lfloor\log_2\left(2^{u+2}\omega\left(j+1\right)\right)\right\rfloor\\
                 \varnothing&\text{ otherwise}
                \end{cases}
\end{align*}
and set $\Gamma=\{j\in \N\colon 2^{j-1}\geq\left\lfloor\log_2\left(2^{u+2}\omega\left(j+1\right)\right)\right\rfloor\}$.
Note here that $\overline{I}_j$ and $\Gamma$ implicitly depend on $\omega$ and $u$.

Assume now that there exists $i\in \overline{I}_{j-1}$ fulfilling $a_i\geq 2^{j+u+5}\cdot\omega\left(j\right)$.
Then, the fact that $\log_2 \left(2^{j+u+5}\cdot\omega\left(j\right)\right)\geq \left\lfloor\log_2\left(2^{u+2}\omega\left(j\right)\right)-\log_2j\right\rfloor-1$
and Lemma \ref{lem: ai+k} imply that we have for all $k\in\left\{i,\ldots,i+\left\lfloor\log_2\left(2^{u+2}\omega\left(j\right)\right)-\log_2j\right\rfloor-1\right\}$
and $x\in \Omega'$ that 
\begin{align}
 a_k
 &\geq 2^{j+u+5-\left(\left\lfloor\log_2\left(2^{u+2}\omega\left(j\right)\right)-\log_2j\right\rfloor-1\right)}\cdot\omega(j)-1
 = 2^{j-\left\lfloor\log_2\left(2^{u+2}\omega\left(j\right)\right)-\log_2j\right\rfloor+u+6}\cdot\omega(j)-1\notag\\
 &\geq 2^{j-\log_2\left(2^{u+2}\omega\left(j\right)\right)-\log_2j+u+5}\cdot\omega(j)-1
 = 2^{j-\log_2j+3}-1\notag\\
 &\geq 2^{j+2-\log_2 j},\label{eq: ak geq}
\end{align}
for $j$ sufficiently large.
We further note that
\begin{align*}
\#\left\{i,\ldots,i+\left\lfloor\log_2\left(2^{u+2}\omega\left(j\right)\right)-\log_2j\right\rfloor-1\right\}
&=\left\lfloor\log_2\left(2^{u+2}\omega\left(j\right)\right)-\log_2j\right\rfloor,
\end{align*}
and 
$\left\{i,\ldots,i+\left\lfloor\log_2\left(2^{u+2}\omega\left(j\right)\right)-\log_2j\right\rfloor-1\right\}\subset I_{j-1}$,
if $i\in \overline{I}_{j-1}$ and $j$ sufficiently large.
Applying this on \eqref{eq: ak geq} 
gives 
\begin{align*}
\MoveEqLeft\left\{\#\left\{i\in I_{j-1}\colon a_i\geq 2^{j+2}\cdot j\right\}\geq \left\lfloor\log_2\left(2^{u+2}\omega\left(j\right)\right)-\log_2 j\right\rfloor\right\}\\
 &\supset \left\{\#\left\{i\in \overline{I}_{j-1}\colon a_i\geq 2^{j+u+5}\cdot\omega\left(j\right)\right\}\geq 1\right\}.
\end{align*}
Combining this with \eqref{eq: number ileq n supset1} gives
\begin{align*}
 \bigcup_{n\in I_j}\left\{\#\left\{i\leq n\colon a_i\geq n\cdot \log n\right\}\geq b_n\right\}
 &\supset \bigcup_{i\in \overline{I}_{j-1}}\left\{ a_i\geq 2^{j+u+5}\cdot\omega\left(j\right)\right\}
\end{align*}
and thus, for $k\in\N$,
\begin{align}
 \bigcup_{n\geq k}\left\{\#\left\{i\leq n\colon a_i\geq n\cdot \log n\right\}\geq b_n\right\}
 &\supset \bigcup_{j\geq \lfloor\log_2 k\rfloor+1}\bigcup_{i\in \overline{I}_{j-1}}\left\{ a_i\geq 2^{j+u+5}\cdot\omega\left(j\right)\right\}.\label{eq: number ileq n supset2}
\end{align}

In the next steps we will make use of the following dynamical Borel-Cantelli lemma.
\begin{lemma}[{\cite[Special case of Theorem 2.1]{kim_dynamical_2007}}]\label{lem: dyn BC-lemma}
Let $\left[0,1\right)$ be partitioned into a finite set of intervals $W_i=\left[c_i,d_i\right)$.
Further, let $\xi\colon \left[0,1\right)\to \left[0,1\right)$ be
such that $\xi$ is derivable on the interior of each $W_i$ and $\xi'\lvert_{\mathring{W}_i}\in BV$. 

Assume that $\xi$ has a uniquely absolutely continuous invariant measure
$\mathrm{d}\mu=h\mathrm{d}\lambda$ and
$h$ is bounded away from $0$. If $A_n$ 
is a sequence of intervals with $\sum_{n=1}^{\infty}\mu\left(A_n\right)=\infty$, then
\begin{align*}
 \lim_{n\to\infty}\frac{\sum_{k=1}^n\mathbbm{1}_{A_n}\circ \xi^{k-1}}{\sum_{k=1}^n\mu\left(A_n\right)}=1~\text{ a.s.}
\end{align*}
\end{lemma}
It follows easily that Lemma \ref{lem: dyn BC-lemma} is applicable. 
The interval $[0,1)$ is partitioned into the intervals $[0,1/2)$ and $[1/2,1)$ and $\tau'\lvert_{(0,1/2)}=\tau'\lvert_{(1/2,1)}=2\cdot\mathbbm{1}\in BV$.
Further, the absolutely continuous measure for this transformation is the Lebesgue measure itself. 

We have by \eqref{eq: F dist func} that
\begin{align*}
\lambda\left(a_i\geq 2^{j+u+5}\cdot\omega\left(j\right)\right)=\frac{1}{\left\lfloor 2^{j+u+5}\cdot\omega\left(j\right)\right\rfloor}.
\end{align*}
Thus,
\begin{align*}
\sum_{j=1}^{\infty}\sum_{i\in\overline{I}_{j-1}}\lambda\left(a_i\geq 2^{j+u+5}\cdot\omega\left(j\right)\right)
&=\sum_{j=1}^{\infty}\sum_{i\in\overline{I}_{j-1}}\frac{1}{\left\lfloor 2^{j+u+5}\cdot\omega\left(j\right)\right\rfloor}
= \sum_{j=1}^{\infty}\frac{\#\overline{I}_{j-1}}{\left\lfloor 2^{j+u+5} \cdot\omega\left(j\right)\right\rfloor}\\
&= \sum_{j\in\Gamma}\frac{\#\overline{I}_{j-1}}{\left\lfloor 2^{j+u+5} \cdot\omega\left(j\right)\right\rfloor}
=\sum_{j\in\Gamma}\frac{2^{j-1}-\left\lfloor\log_22\omega\left(j\right)\right\rfloor}{\left\lfloor 2^{j+u+5}\cdot\omega\left(j\right)\right\rfloor}.
\end{align*}
By the definition of $\Gamma$ we have for $j\in \Gamma$ that $2^{j-1}-\left\lfloor\log_2\left(2^{u+2}\omega\left(j\right)\right)\right\rfloor\geq 2^{j-2}$
yielding
\begin{align}
\sum_{j=1}^{\infty}\sum_{i\in\overline{I}_{j-1}}\lambda\left(a_i\geq 2^{j+u+5}\cdot\omega\left(j\right)\right)
&\geq \sum_{j\in\Gamma}\frac{1}{2^{u+7}\cdot\omega\left(j\right)}.\label{eq: sum 1/kappa psi}
\end{align}
Furthermore, we note that $j\notin\Gamma$ implies $\omega(j)>2^{2^{j-2}}\cdot 2^{-u-2}>2^{2^{j-3}}$,
for $j$ sufficiently large, say larger than $J\in \N$.
Thus,
\begin{align*}
 \sum_{j\notin\Gamma\cup\N_{\leq J}}\frac{1}{\omega(j)}
 < \sum_{j\notin\Gamma\cup\N_{\leq J}}\frac{1}{2^{2^{j-2}}}
 <\sum_{j=1}^{\infty}\frac{1}{2^{2^{j-2}}}
 <\infty
\end{align*}
implying that $\sum_{j\in\Gamma}1/\omega(j)=\infty$ since $\omega\in\overline{\Psi}$.
Combining this consideration with \eqref{eq: sum 1/kappa psi} yields
\begin{align*}
 \sum_{j=1}^{\infty}\sum_{i\in\overline{I}_{j-1}}\lambda\left(a_i\geq 2^{j+u+5}\cdot\omega\left(j\right)\right)
 =\infty.
\end{align*}
Since 
\begin{align*}
\left\{a_i\geq 2^{j+u+5}\cdot\omega\left(j\right)\right\}=\left\{\mathbbm{1}_{\left[0,\lceil 2^{j+u+5}\cdot\omega\left(j\right)\rceil^{-1}\right)}\circ \tau^{i-1}= 1\right\}, 
\end{align*}
where $\left\lceil x\right\rceil=\min\left\{n\in\mathbb{Z}\colon n\geq x\right\}$,
the conditions of Lemma \ref{lem: dyn BC-lemma} are fulfilled and 
we can apply the second Borel-Cantelli lemma proving 
\begin{align*}
\lambda\left(\bigcap_{n\in\mathbb{N}}\bigcup_{j\geq n}\bigcup_{i\in\overline{I}_j}\left\{a_i\geq 2^{j+u+5}\cdot \omega\left(j\right)\right\}\right)=1.
\end{align*}
Combining this with \eqref{eq: number ileq n supset2} gives \eqref{eq: ai>nlogn io} and thus the statement of the theorem. 
\end{proof}

\section*{Acknowledgements}
The author would like to thank Alan Haynes for helpful discussions and comments.

\end{document}